\documentclass[a4paper,12pt]{amsart}
\usepackage{amsmath,amssymb,amsfonts,amsthm,exscale,calc}
\usepackage[latin1]{inputenc}
\usepackage{latexsym,textcomp}
\usepackage{graphicx,epsf}
\usepackage{dsfont}
\usepackage[german, english]{babel}

\usepackage[dvips, bookmarks, colorlinks=false, breaklinks=true]{hyperref}

\newtheorem{lemma}{Lemma}[section]
\newtheorem{theorem}[lemma]{Theorem}

\newtheorem{proposition}[lemma]{Proposition}
\newtheorem{corollary}[lemma]{Corollary}

\newtheorem{problem}{Problem}

\theoremstyle{definition}
\newtheorem{definition}[lemma]{Definition}
\newtheorem{example}[lemma]{Example}

\theoremstyle{remark}

\numberwithin{equation}{section}
\newcommand{\comment}[1]{}

\newcommand{\R}{{\mathbb R}}
\newcommand{\C}{{\mathbb C}}

\newcommand{\N}{{\mathbb N}}

\newcommand{\A}{{\mathcal{A}}}
\newcommand{\LL}{{\mathcal{L}}}

\newcommand{\Lip}{C_{\mbox{\scriptsize{Lip}}}}
\newcommand{\diam}{{\mathrm {diam}\,}}

\newcommand{\cpt}{{canonically compactifiable}}

\newcommand{\al}{{\alpha}}

\newcommand{\de}{{\delta}}
\newcommand{\eps}{{\varepsilon}}

\newcommand{\ph}{{\varphi}}

\newcommand{\gm}{{\gamma}}
\newcommand{\si}{{\sigma}}
\newcommand{\hrt}{{\heartsuit}}

\newcommand{\as}[1]{\left\langle #1\right\rangle}

\newcommand{\aV}[1]{\left\Vert #1\right\Vert}
\newcommand{\ov}[1]{\overline{ #1}}
\newcommand{\ow}[1]{\widetilde{ #1}}
\newcommand{\oh}[1]{\widehat{ #1}}

\newcommand{\Hm}[1]{\leavevmode{\marginpar{\tiny%
$\hbox to 0mm{\hspace*{-0.5mm}$\leftarrow$\hss}%
\vcenter{\vrule depth 0.1mm height 0.1mm width \the\marginparwidth}%
\hbox to 0mm{\hss$\rightarrow$\hspace*{-0.5mm}}$\\\relax\raggedright
#1}}}

\begin{document}

\title[Graphs of finite measure]{Graphs of finite measure}

\author[Georgakopoulos]{Agelos Georgakopoulos} \address{Mathematics Institute, Zeeman Building \\ University of Warwick \\ Coventry CV4 7AL} \email{a.georgakopoulos@warwick.ac.uk}

\author[Haeseler]{Sebastian Haeseler} \address{Mathematisches Institut \\ Friedrich Schiller Universit{\"a}t Jena \\ 07743 Jena, Germany } \email{sebastian.haeseler@uni-jena.de}

\author[Keller]{Matthias Keller} 
\address{Einstein Institute of Mathematics, The Hebrew University of Jerusalem, Jerusalem 91904, Israel}\email{mkeller@ma.huji.ac.il}

\author[Lenz]{Daniel Lenz} \address{Mathematisches Institut \\Friedrich Schiller Universit{\"a}t Jena \\07743 Jena, Germany } \email{daniel.lenz@uni-jena.de}

\author[Wojciechowski]{Rados{\l}aw K. Wojciechowski}\address{York College of the City University of New York \\ Jamaica, NY 11451 \\ USA  } \email{rwojciechowski@gc.cuny.edu}

\date{\today}

\begin{abstract}
We consider  weighted graphs  with an infinite set of vertices. We show that boundedness of all functions of finite energy can be seen as  a notion of  `relative compactness'  for such graphs  and study sufficient and necessary conditions for  this property in terms of various  metrics. We then equip graphs satisfying  this property  with a finite measure and investigate   the associated Laplacian and its semigroup. In this context, our results include the trace class property for the semigroup, uniqueness and existence of solutions to  the Dirichlet problem with boundary arising from the natural compactification,  an explicit description of the domain of the Dirichlet Laplacian,   convergence of the heat semigroup for large times as well as stochastic incompleteness and transience of the corresponding random walk in continuous time.  \end{abstract}

\maketitle

\tableofcontents


\section{Introduction}

The study of the spectral and potential theory  of infinite graphs has been a flourishing subject in recent years. In particular, the case of infinite graphs with possibly unbounded vertex degree has received substantial attention, e.g.,  \cite{BJH,CdVTHT, Geo2,Geo,GK,HKLW, Hua1, HKMW, Jor, JP,JP2, JL,KL1,KL2, KLVW, LyonsBook, TH, Woj1, Woj2} and references therein.   In this paper, we want to contribute to this topic by having a closer look at a remarkable subclass of infinite graphs, namely, those which are `relatively compact.'

Indeed, the  subsequent considerations can be divided into two parts. In the first part, we present a class of graphs which have many claims to be considered as relatively compact in a natural sense.  In the second part, we then study spectral, probabilistic and semigroup aspects of these graphs.

Let us discuss this in somewhat more detail and with the notation introduced and discussed in later sections. In fact,  in order to provide  here  a concise overview,  we will  refrain  from giving precise definitions but  just indicate at which places they can be found in the paper.

We consider a countably infinite set $X$ together with a graph structure  $(b,c)$. Here, $b:X\times X\longrightarrow [0,\infty)$ is symmetric, vanishes on the diagonal and satisfies  the summability condition
$$\sum_{y\in X} b(x,y) <\infty$$
for all $x\in X$.  The function $c:X \to [0,\infty)$ is often assumed to be zero in the literature. Here, we mostly do not need $c$ to vanish but can rather deal with summable $c$ or even arbitrary $c$.  A graph $(b,c)$ over $X$ naturally gives rise to a form $\ow Q$ on the domain $\ow D$ of functions of finite energy and the  associated operator $\LL$ (Section~\ref{section-setup}).

\medskip

Our aim in  the first part of the paper  is to capture a situation in which the graph  $(b,c)$ over $X$  can be considered to be relatively compact. Various (classes of)  metrics have been considered on graphs (Section \ref{section-analysis}).  First, there is a natural length structure, using the inverse of the edge weights $b$,  giving rise to the metric $d$.  This metric, and its completion, has found various applications, some of which are described in \cite{Geo}.  It has also recently  appeared  in the study of spectral properties, e.g., \cite{CdVTHT,HKLW,TH}. Another relevant metric is $\varrho$. It is  the square root of the well-known free effective resistance metric $r$ (as shown below in Theorem~\ref{t:rhofree}). The free effective resistance plays an important role in the investigation of networks, see, e.g., the remarkable study of Kigami \cite{Kig} or the influential textbook of Lyons/Peres  \cite{LyonsBook} and recent work of Jorgensen/Pearse  \cite{JP2}.  Here, we focus on $\varrho$ rather than $r$, as it is very close in spirit to concepts found in  both the geometry of Dirichlet spaces and  non-commutative geometry. Indeed, it has already been studied in the context of graphs and non-commutative geometry by Davies in \cite{Davies2} and, very recently, it has appeared in  a similar spirit in \cite{HKT2}.
 Finally, there are the so-called intrinsic metrics $\sigma$.
They have been brought forward and  first studied systematically in \cite{FLW} (see \cite{Uem} for earlier appearances as well). Subsequently, they have  become a main tool in certain geometric and spectral theoretic considerations, see, e.g., \cite{BHK, BKW,Fol,Fol2,GHM,HKW,HuK,MU,MUW}.   They are always defined with respect to some
measure $m$ on $X$. In our situation, this measure should be finite as we want to deal with a relatively compact case.

Given this situation, each of  the following
statements seems to be a good candidate to express the relative compactness of the graph $(b,c) $ over $X$:

\begin{itemize}

\item[$(A)$] $X$ is totally bounded with respect to $d$.
\item[$(B)$]  $X$ is totally bounded with respect to $\varrho$.
\item[$(C)$]  $X$ is totally bounded with respect to any  metric $\sigma $ which is intrinsic with respect to a finite measure.
\end{itemize}
While these are certainly sensible  requirements, they have the disadvantage that $(A)$ and $(B)$ are rather strong and $(C)$ is hard to get a grip on.

On the other hand, letting $\ow D$ denote the functions of finite energy, there is  the - at first seemingly - unrelated notion:

\begin{itemize}
\item[$(D)$]  $\ow D$ consists only of bounded functions.
\end{itemize}
In this case, there  naturally appears a compactification  $K=X\cup\partial X$  of $X$ such that the closure of $\ow D$ can be realized canonically as an algebra of continuous functions on $K$ (Subsection \ref{secBoCo}). This compactification $K$ is obtained by $C^\ast$-algebra techniques and turns out to coincide with the Royden compactification $R$ (Theorem~\ref{KR}).
Now, one of the main results of the first part of the paper shows that (when $c \equiv 0$) the implications
$$ (A)\Longrightarrow (B)\Longrightarrow (C) \Longrightarrow (D)$$
hold (Theorem \ref{implication-D}).
Moreover, examples show that the implications are essentially strict. More precisely, as shown by examples in   Section~\ref{counter},   neither does $(B)$ imply $(A)$ nor does $(C)$ imply  $(B)$.
As it is, therefore,
both the most general and also a completely canonical way of phrasing the relative  compactness of $X$, we single out graphs satisfying $(D)$ as the \textit{canonically compactifiable} ones.

 Our corresponding considerations in Sections \ref{section-analysis} and \ref{section-topology} also show that there are  unique continuous maps extending the identity on $X$ giving rise to the following chain of maps:
$$ \ov{X}^d \stackrel{\iota}{\longrightarrow} \ov{X}^\varrho \stackrel{\kappa}{\longrightarrow} K=R \stackrel{\lambda}{\longrightarrow} \ov{X}^\sigma.$$
This (and more) is summarized in Theorem \ref{implication-D}.
While there is no reason why these maps should be homeomorphisms in general, there is one prominent situation in which they are, namely, if $c\equiv0$ and $1/b$ is summable, see Section~\ref{key-example}.

Along the way, we also give a characterization of $\varrho$ via intrinsic metrics  in the case $c\equiv0$ (Theorem \ref{characterization-varrho-via-sigma})  as well as various necessary and/or sufficient conditions for validity of the conditions $(A)$, $(B)$, $(C)$ and $(D)$ (Theorem \ref{diam-d-finite-implies-C}, Theorem \ref{characterization-cpt}, Corollary  \ref{sufficient-condition-cpt-d} and Corollary \ref{sufficient-condition-cpt-intrinsic}).

The preceding  considerations are the content of Sections~\ref{section-analysis}  and \ref{section-topology}.   They give our treatment of what it means for a graph to be relatively compact.   This is the first part of the paper.

\medskip

The second part then deals with spectral, probabilistic and semigroup  features of graphs satisfying $(D)$.

\smallskip

More specifically, in Section~\ref{Section-Operator}, we  first study  selfadjoint Laplacians arising when  graphs satisfying $(D)$ are equipped with a finite measure. These operators are shown to have purely discrete spectrum. In fact, their resolvents and semigroups are shown to even be trace class (Section~\ref{Discrete}). We then study the Dirichlet problem on graphs satisfying $(D)$ only, with no assumptions on the measure.  Here, we show  (in Section~\ref{Boundary}) that, for any continuous function $\varphi$ on $ \partial X$, there exists a unique $f\in \ow D$ satisfying
$$\LL  f = 0\;\:\mbox{on $X$ and }\: f|_{\partial X} = \varphi.$$

\smallskip

In  Section~\ref{Section-Geometry}, we then study convergence of the heat kernels for times tending to $\infty$.

\smallskip

Finally, in Section~\ref{Section-Probability}, we recall a result of \cite{Schm} giving  that for graphs $(b,c)$ with $c\equiv0$    over $X$ equipped with a finite measure stochastic completeness, recurrence and the equality $Q^{(D)} = Q^{(N)}$ are all equivalent. We then use this to show that none of these properties hold  if the graph additionally satisfies $(D)$.

\medskip

We are not aware of any work related to ours in the spirit of studying infinite, relatively compact graphs.  However, of course,  in terms of methods and intermediate results
there are certainly  points of contact to various works. This is discussed in the corresponding places in the paper. Here, we already want to mention the following:

The idea of embedding a graph in a compact space $K$ via $C^\ast$-algebras (as done in our considerations below) has quite some history. It goes back at least to the work of Yamasaki \cite{Yam} and Kayano and Yamasaki \cite{KY}  (see the discussion in  \cite{Soa} as well). However, the focus of these works is completely different from ours. In these works, the corresponding constructions are carried out for arbitrary graphs (satisfying the assumption  $c\equiv0$), whereas we use it to single out the special class of graphs where $\ow D$ consists of bounded functions. For a recent treatment  of how to  embed  a general Dirichlet space into a locally  compact one, we refer the reader to \cite{HKT}.

Also, there is work of Carlson on the Dirichlet Problem, both for discrete graphs and for  selfadjoint continuum Laplacians on  metric graphs satisfying certain compactness type properties, which is somewhat similar to our considerations on the Dirichlet Problem \cite{Car,CarAft} (see corresponding section for more detailed discussion).

\medskip

\textbf{Acknowledgements.} {\small D.L. and  M.K. gratefully acknowledge partial support from the German Research  Foundation (DFG).
D.L. also takes the opportunity to thank Marcel Schmidt and Peter Stollmann for most useful discussions on a wide array of topics related to the present work and A. Teplyaev for pointing out some  references. A substantial  part of this work was done while D.L. was visiting   the Department of Mathematics of Universit\'{e} Lyon 1 and the Department of Mathematics of Geneva University  and he expresses his warmest thanks to these institutions. M.K. furthermore acknowledges the financial support of the ESF for the research stay at TU Graz where part of this work was done, the Golda Meir Fellowship, the Israel Science Foundation (grant no. 1105/10 and  no. 225/10).  R.W. would like to thank J\'ozef Dodziuk for insightful discussions and acknowledges financial support provided by PSC-CUNY Awards, jointly funded by The Professional Staff Congress and The City University of New York.  Furthermore, M.K., D.L., and R.W. would like to thank the organizers of the LMS EPSRC Durham Symposium on ``Graph Theory and Interactions'' and the generous hospitality of the Department of Mathematical Sciences at Durham University where the (almost) final parts of this  work were carried out.}


\section{The basic  set up: Weighted graphs, forms and Laplacians}\label{section-setup}
In this section, we introduce the basic set up for our considerations.  This includes a discussion of weighted graphs as well as of the associated forms and operators.


\subsection{Weighted graphs} \label{secWG}
Let $X$ be an infinitely countable set.  
A \emph{weighted graph} over $X$ is a pair $(b,c)$ such that $b:X\times X\longrightarrow[0,\infty)$ is symmetric, has zero diagonal, and satisfies
\begin{align*}
    \sum_{y\in X}b(x,y)<\infty
\end{align*}
for all $x\in X$ and $c:X\longrightarrow[0,\infty)$ is arbitrary. We call $X$ the \emph{vertex set}, $b$ the \emph{edge weight} and $c$ the \emph{killing term} or \emph{potential}. We say that $x,y\in X$ are \emph{neighbors} or \emph{connected} by an edge of weight $b(x,y)$, if $b(x,y)>0$.  
If the number of neighbors of each vertex is finite, then we call $(b,c)$ or $b$ \emph{locally finite}. A finite sequence $(x_{0},\ldots,x_{n})$ of pairwise distinct vertices such that $b(x_{i-1},x_{i})>0$ for $i=1,\ldots,n$ is called a \emph{path} from $x_0$ to $x_n$.
We say that $(b,c)$ or $b$ is \emph{connected} if, for every two distinct vertices $x,y\in X$, there is a path from $x$ to $y$.

 All of our main results will assume that $(b,c)$ is connected.  This is mostly for convenience. By extending definitions properly and restricting attention to connected components, we could deal with the more general situation.

Often we will not only be given an infinite countable set $X$ together with a  weighted graph $(b, c)$ but also a function  $m:X\longrightarrow[0,\infty)$.  This function gives naturally rise to a measure on $X$ via $$m(A) := \sum_{x\in A} m(x)$$
for $A\subseteq X$  and we will not distinguish it from this measure.
Then, $(X,m)$ can be considered as a measure space and we will speak about graphs $(b,c)$ over $(X,m)$. For our applications to spectral theory we will then often  assume that $m$ does not vanish on any point. Moreover, the total mass
$$m(X) = \sum_{x\in X} m(x)$$
of $X$
will appear prominently in various parts of our discussion.


\subsection{Formal Laplacians}\label{Subsection-Formal}
Let $C(X)$ be the space of complex valued functions on $X$ and let $C_{c}(X)$ be the subspace of functions with finite support.  Let $\ell^{2}(X,m)$ be the complex Hilbert space of square summable functions equipped with the scalar product
\begin{align*}
    \as{u,v}_m=\sum_{x\in X}\overline{u(x)}v(x) m(x).
\end{align*}
We may  drop the subscript $m$ and also write $\as{u,v}=\sum_{X}\overline{u}v m$ and denote the corresponding norm by $\aV{\cdot}$.  We denote the space of bounded functions on $X$  by $\ell^{\infty}(X)$ and denote by $\aV{\cdot}_{\infty}$ the supremum norm.

Given a weighted graph $(b,c)$ over  $X$ we introduce the associated  \emph{formal Laplacian} $\LL$  acting on
\begin{align*}
\ow F=\{f\in C(X) : \sum_{y\in X}b(x,y)|f(y)|<\infty\mbox{ for all }x\in X\}
\end{align*}
as
\begin{align*}
\LL f (x) = \sum_{y\in X}b(x,y)(f(x)-f(y))+ c(x ) f(x).
\end{align*}
If there is, furthermore, a positive measure $m$ on $X$ at our disposal, we will also consider the operator $\ow L = \frac{1}{m}  \LL$ acting as
\begin{align*}
\ow L f(x)=\frac{1}{m(x)}\sum_{y\in X}b(x,y)(f(x)-f(y))+\frac{c(x)}{m(x)}f(x).
\end{align*}
We can think of  $\ow L$ and $\LL$ as  discrete analogues of the Laplace Beltrami operator on a Riemannian manifold plus a potential. We will be interested in properties of  $\ow L$  and of its selfadjoint realizations. Note that $f\in \ow F$ satisfies  $\LL f = 0$ or $\LL f\leq 0$ or $\LL f \geq 0 $ if and only if it satisfies the corresponding statements with $\LL$ replaced by $\ow L$ for any positive measure $m$.


\subsection{Quadratic forms}\label{Quadratic}
Given a weighted graph $(b,c)$ over  $X$ we define the \emph{generalized form} $\ow Q: C(X)\longrightarrow[0,\infty]$ by
\begin{align*}
\ow Q(f):=\frac{1}{2}\sum_{x,y\in X}b(x,y)|f(x)-f(y)|^{2}+\sum_{x\in X}c(x)|f(x)|^{2}
\end{align*}
and define the \emph{generalized form domain} by
\begin{align*}
    \ow D:=\{f\in C(X) :  \ow Q(f)<\infty\}.
\end{align*}
We think of $\ow Q (f)$ as the energy of the function $f$ and, accordingly,
functions in $\ow D$ are said to have \emph{finite energy}.
Clearly,  $C_{c}(X)\subseteq \ow D$ as $b(x,\cdot)$ is summable for every $x\in X$.

Since $\ow Q^{1/2}$ is a seminorm and satisfies the parallelogram identity
\begin{align*}
\ow Q(f+g)+\ow Q(f-g)=2(\ow Q(f)+\ow Q(g)),\quad f,g\in\ow D,
\end{align*}
it gives, by polarization, a semi scalar product on $\ow D$ via
\begin{align*}
 \ow Q(f,g)=\frac{1}{2}\sum_{x,y\in X} b(x,y)\overline{(f(x)-f(y))}(g(x)-g(y))+\sum_{x\in X}c(x)\overline{f(x)}g(x).
\end{align*}
In the case when $c\not\equiv 0$ and $b$ is connected, the form $\ow Q$ defines a scalar product.

Obviously, $\ow Q$ is compatible with normal contractions in the sense that $\ow Q(C f)\leq \ow Q (f)$ holds for any $f\in C(X)$ and any normal contraction $C : \C \longrightarrow \C$. (Here,  $C : \C \longrightarrow \C $ is a \textit{normal contraction} if both $|C (p) |\leq |p|$ and $|C(p) - C(q)| \leq |p - q|$ hold for all $p,q\in \C$. )

We now assume that we are additionally given a measure $m$ on $X$ of full support. Then, suitable restrictions of $\ow Q$ are in correspondence with certain selfadjoint restrictions of $\ow L$ on the Hilbert space $\ell^2 (X,m)$. This is discussed next:

Let  $Q$ be a closed non-negative form on $\ell^2 (X,m)$ whose domain $D=D(Q)$ satisfies $C_{c}(X)\subseteq D\subseteq \ow D\cap \ell^{2}(X,m)$ and which satisfies
\begin{align*}
Q(u,v)=\ow Q(u,v),
\end{align*}
for $u,v\in D$. For such a form we define $Q(u):=Q(u,u)$ for $u\in D$ and $Q(u):=\infty$ for $u\not\in D$.

 In \cite{HK,HKLW} an `integration by parts'  was shown that allows one to pair functions in $\ow F$ and $C_{c}(X)$ via $\ow Q$.
  More precisely,  for $f\in \ow F$ and $v\in C_{c}(X)$ it was shown that
\begin{align*}
  \frac{1}{2} \sum_{x,y\in X}  & b(x,y)\ov{(f(x)-f(y))}(v(x)-v(y))+\sum_{x\in X}c(x)\ov{f(x)}v(x)\\
&=\sum_{x\in X}\ov{f(x)}(\ow L v)(x)m(x) = \sum_{x\in X} \ov{(\ow L f)(x)} v(x)m(x),
\end{align*}
where all  sums converge absolutely. Moreover, it is shown there that  $\ow D$ is a subset of $\ow F$ and that for $f \in \ow D$ the preceding sums all agree with $\ow Q (f,v)$. Note that this immediately gives
\begin{align*}
 \frac{1}{2} \sum_{x,y\in X} & b(x,y)\ov{(f(x)-f(y))}(v(x)-v(y))+\sum_{x\in X}c(x)\ov{f(x)}v(x)\\
&=\sum_{x\in X}\ov{f(x)}(\LL v)(x) = \sum_{x\in X} \ov{(\LL  f)(x)} v(x),
\end{align*}
for $f\in \ow F$ and $v\in C_c (X)$.

By \cite[Proposition~3.3]{HKLW} we then have that the self adjoint operator $L$ associated to  a form $Q$ as above satisfies
\begin{align*}
Lu=\ow Lu
\end{align*}
for all $u\in D(L)$.


There are two natural examples, $Q^{(N)}$ and $Q^{(D)}$, which we refer to as the \emph{forms with Neumann} and \emph{Dirichlet boundary conditions}, respectively.
The form $Q^{(N)}$ has the domain
\begin{align*}
D(Q^{(N)})=\ow D\cap \ell^{2}(X,m)=\{u\in\ell^{2}(X,m) :  \ow Q(u)<\infty\}.
\end{align*}
The form $Q^{(D)}$ has the domain
\begin{align*}
D(Q^{(D)})=\overline{C_{c}(X)}^{\aV{\cdot}_{\ow Q}},
\end{align*}
where
\begin{align*}
\aV{u}_{\ow Q}:={(\ow Q(u)+\aV{u}^{2})}^{\frac{1}{2}}.
\end{align*}
We denote the corresponding self adjoint operators by $L^{(N)}$ and $L^{(D)}$.

 By construction, both $Q^{(N)}$ and $Q^{(D)}$ are Dirichlet forms as $\ow Q$ is  compatible with normal contractions.
Moreover, $Q^{(D)}$ is regular, i.e., $D(Q^{(D)})\cap C_{c}(X)$ is dense in $D(Q^{(D)})$ with respect to ${\aV{\cdot}}_{\ow Q}$ and in $C_{c}(X)$ with respect to $\aV{\cdot}_{\infty}$.

The following basic lemma shows that all forms in question are sandwiched by the forms with Dirichlet and Neumann boundary conditions.  Here, we write $Q_{1}\subseteq Q_{2}$ if $D(Q_{1})\subseteq D(Q_{2})$ and $Q_{1}(u)=Q_{2}(u)$ for $u\in D(Q_{1})$.

\begin{lemma} If $Q$ is a symmetric, closed quadratic form with domain $D$, then $Q^{(D)}\subseteq Q\subseteq Q^{(N)}$ if and only if $C_{c}(X)\subseteq D\subseteq\ow D\cap\ell^{2}(X,m)$ and $Q(u)=\ow Q(u)$ for all $u\in D$.
\end{lemma}
\begin{proof} The `only if' part of the statement  is clear since $C_{c}(X)\subseteq D(Q^{(D)})$, $D(Q^{(N)})=\ow D\cap\ell^{2}(X,m)$ and $Q^{(N)}(u)=\ow Q(u)$ for $u\in D(Q^{(N)})$. For the `if' part of the statement  notice that $D(Q^{(D)})\subseteq D$ since $Q$ is closed and $C_c(X)\subseteq D$. Moreover,  $D\subseteq\ow D\cap\ell^{2}(X,m)=D(Q^{(N)})$ and $Q^{(D)}(u)=\ow Q(u)=Q(u)$ for $u\in D(Q^{(D)})$.
\end{proof}



\textbf{Remark.}  Let us shortly comment on the role of $c$ in our considerations. In  the study of graphs, one usually sets $c\equiv0$, i.e., one  considers just a function  $b$ as defined above on a countable set $X$. Here, our point of view is different. We study regular Dirichlet forms on the measure space $(X,m)$. These forms are naturally in one-to-one correspondence with graphs $(b,c)$ as discussed in \cite{KL1}. From this point of view, the appearance of $c$ is quite natural. For the subsequent considerations, two regimes of $c\not \equiv 0$ can be distinguished. If $0 < \sum_{x\in X} c(x) < \infty$, one can add  a point $\infty$ to the graph and connect it to $x\in X$ with an edge of weight $c(x)$. In this way, one obtains a new graph without $c$ and with one additional point. The relevant set of functions then consists of those vanishing on this additional point. If $\sum_{x\in X} c(x) = \infty$, such a construction is not possible. A more detailed description of this is given in Appendix \ref{Reducing}.


\section{Analysis: The metrics $\varrho$ and $d$  and intrinsic metrics}\label{section-analysis}
The subsequent considerations deal with various aspects of metrics on graphs over discrete sets. The interplay between properties of the metric and the graph structure is the main focus.

\bigskip

We start by recalling some basic notation on metric spaces used throughout the section:
As usual, for a space $X$ with a metric $\delta$ we denote the completion of $X$ with respect to $\delta$ by $\ov{X}^\delta$. Obviously, $\delta$ can then be extended to a metric on $\ov{X}^\delta$ which will be denoted by the same letter.

Also, we define  the distance to  a set $A\subset \ov{X}^\delta$ by
$$\delta_A : X\longrightarrow [0,\infty), \: \; \delta_A (x) := \inf\{ \delta(x,a): a\in A\}.$$
The diameter of  $X$ with respect to $\delta$ is given by
$$\mbox{diam}_\delta (X) := \sup\{\delta (x,y) : x,y\in X\}.$$
If the metric $\delta$ is understood from the context, we only write $\mbox{diam} X $ for $\diam_\delta (X)$.

The set of Lipschitz continuous functions  on a metric space $X$ is denoted by $\Lip (X)$, where the dependence on the metric is suppressed.

As usual, a metric space is called \textit{totally bounded} if, for any $\varepsilon >0$, it can be covered by finitely many balls of radius $\varepsilon$. A metric space is compact if and only if it is totally bounded and complete.

In certain cases, it will also be convenient for us to be able to work with pseudometrics.
A pseudometric  on  a set $X$ is a function that satisfies all of the conditions of a metric, except that it may  be zero outside of the diagonal. Whenever $\delta$ is a pseudometric on a set $X$,
it will induce a metric $\widetilde{\delta}$  on the set $\widetilde{X}$ which is the quotient of $X$ arising from identifying points whose distance is zero. By a slight abuse of notation, we then denote the completion of $\widetilde{X}$ with respect to  $\widetilde{\delta}$ by $\ov{X}^\delta$ and refer to it as  the completion of $X$ with respect to $\delta$.


\subsection{The metrics $d$ and $\varrho$}\label{Metrics}
In this subsection, we discuss two well-known metrics on graphs.  The metric $d$ has featured prominently in the work of one of the authors \cite{Geo}  and has also played a role in the spectral theoretic considerations of \cite{CdVTHT,TH}.
The metric $\varrho$ seems to be less known. It has played some role in work of Davies \cite{Davies2}. It is a very natural object from both the point of view of non-commutative geometry and that of Dirichlet forms (see remark below for further discussion).
\bigskip

In this subsection, we consider a
weighted graph   $(b,c)$ over the discrete set $X$. Let us emphasize that our considerations do not require any measure $m$ on $X$.

We fix a vertex $o\in X$ and define a semi scalar product $\as{\cdot,\cdot}_{o}$ on $\ow D$ by
\begin{align*}
\as{f,g}_{o}=\ow Q(f,g)+\ov{f(o)}g(o),
\end{align*}
for $f,g\in \ow D$ and the corresponding semi norm
\begin{align*}
    \aV{f}_{o}:={\as{f,f}_{o}}^{\frac{1}{2}}=
    {(\ow Q(f)+|f(o)|^{2})}^{\frac{1}{2}}.
\end{align*}
If $b$ is connected, then $\as{\cdot,\cdot}_{o}$ defines a scalar product and $\aV{\cdot}_{o}$ defines a norm on $\ow D$. We call $\ow D$ equipped with $\as{\cdot,\cdot}_{o}$ the \emph{Yamasaki space} after \cite{Yam}, where a first study of this space was undertaken (see later work of Soardi \cite{Soa} for a systematic  investigation as well).

Assume that $b$ is connected. We define the  functions   $\varrho$ and $\varrho_{o}$ on $X\times X$ by
\begin{align*}
\varrho(x,y)&:=\sup\{|f(x)-f(y)| : f\in \ow D,\ow Q({f})\leq1\}, \\
\varrho_{o}(x,y)&:=\sup\{|f(x)-f(y)| : f\in \ow D,\aV{f}_{o}\leq1\},
\end{align*}
for $x,y\in X$.
It is not hard to see that both $\varrho$ and $\varrho_o$ are metrics. In fact, every defining feature of a metric is  clear except for the finiteness of the values. This finiteness can easily be inferred from connectedness (an explicit bound is given below in Lemma \ref{l:DLip}).
Clearly, all functions in $\ow D$ are Lipschitz continuous with respect to $\varrho$ and $\varrho_{o}$ (with Lipschitz constant bounded by  $\ow Q^{\frac{1}{2}} (f)$ and $\aV{f}_o$, respectively).  The metric $\varrho$  is essentially the smallest metric making all elements in $\ow D$ Lipschitz continuous with constant given by the respective values of the form. \\

\textbf{Remarks.}(a) Using the material presented below, it  can be  shown that in the definitions of $\varrho$ and $\varrho_0$ the supremum can  be replaced by a maximum. Indeed, the argument for $\varrho_o$ can be given as follows: Choose for $x,y\in X$ a sequence $f_n$ with $\|f_n\|_o \leq 1$ and  $\varrho_o (x,y) = \lim_{n\to \infty} |f_n (x) - f_n (y)|$. Then, by Lemma \ref{l:pointevaluation}, the sequence $(f_n (p))$ must be bounded for any $p\in X$. By a standard diagonal sequence argument, we can then assume that $(f_n)$ converges pointwise to a function $f$. This gives, in particular,
$$\varrho_o (x,y) = \lim_{n\to \infty}  |f_n (x) - f_n (y) | = |f(x) - f(y)|.$$
 Moreover, by Fatou's lemma and the pointwise convergence, we have $\ow Q (f) + |f(o)|^2 \leq \liminf_{n\to \infty} ( \ow Q (f_n) + |f_n(o)|^2 ) \leq 1$.
  Thus, the desired statement follows.

 Note that the crucial ingredient in the above reasoning is to show pointwise boundedness of the sequence $(f_n)$.  If $c\not \equiv 0$, essentially the same argument works for $\varrho$ (with the role of $o$ played by any point $p$ with $c(p) >0$). If $c\equiv0$, then each of the $f_n$ can be shifted by a constant. In particular, without loss of generality one can assume that $f_n (o) = 0$ for all $n$. Now, the argument can be concluded as in the case of $\varrho_o$ (compare Proposition \ref{p:rho} as well).

  (b) In Subection~\ref{seceffres}, we express $\varrho$ in the terminology of electrical networks: it turns out that its square equals the free effective resistance. In this sense, the metric $\varrho$ is essentially a variant of a well-known quantity in graph theory. From our point of view, it has two structural  advantages over the free effective resistance: First of all, it is closely tied to the intrinsic metrics discussed in the next chapter, see, e.g., Theorem \ref{characterization-varrho-via-sigma}. Intrinsic metrics, in turn, are a key concept in the study of the geometry of Dirichlet forms.  Moreover, its definition is  close in spirit to non-commutative geometry. In fact, it has been considered in the context of non-commutative geometry and graphs by Davies  in  \cite{Davies2} and also, very recently, by Hinz / Kelleher / Teplyaev \cite{HKT2}.

\medskip


Obviously, $\varrho_{o} \leq \varrho$ by definition. While the opposite inequality is wrong,  it turns out that, in general, these two metrics are comparable, i.e., there is no need to distinguish between $\varrho$ and $\varrho_{o}$. This is the content of the next proposition.

\begin{proposition}\label{p:rho} If $(b,c)$ is connected, then there exists $o\in X$ such that $\varrho$ and $\varrho_{o}$ are equivalent as metrics. Moreover, $\varrho \equiv \varrho_{o}$ whenever $c\equiv0$.
\end{proposition}
\begin{proof}Clearly, $\varrho_{o}\leq\varrho$.

If $c\equiv0$, then let $o\in X$ be arbitrary. For every $f\in\ow D$ with $\ow Q(f)\leq1$  we have $(f-f(o))\in\ow D$ and $\aV{f-f(o)}_{o}^2=\ow Q(f)$. Hence, $\varrho\leq\varrho_{o}$.

If $c\not\equiv0$, then choose  $o\in X$ with $c(o)>0$. Then $\aV{f}_{o}^{2}\leq(1+\frac{1}{c(o)})\ow Q(f)$. Therefore, $\varrho\leq C\varrho_{o}$ for some $C>0$.
\end{proof}

 For later use we also introduce the following notation.
\begin{definition} \label{def:Dzero} Let $(b,c)$ be a connected graph over $X$ and let $o\in X$ be fixed.  Define
$\ow D_o$ to be   the closure of $C_c (X)$ in $\ow D$ with respect to $\aV{\cdot}_o$.
\end{definition}

We introduce another family of metrics that emerge more directly from the geometry of the graph.

\begin{definition}\label{length_function} Let $(b,c)$ be a graph over $X$. A function $\ell : X\times X\longrightarrow [0,\infty)$ is called a \emph{length function} if, for $x,y\in X$, the equality  $\ell (x,y) = 0$ holds  if  $b(x,y) = 0$ holds.
Whenever $\ell$ is a length function and
$\gm=(\gm_{0},\ldots,\gm_{n})$ is  a path in $X$, the \emph{$\ell$-length} $\ell (\gm)$ of $\gm$ is defined  to be
\begin{align*}
\ell (\gm):=\sum_{i=1}^{n} \ell (\gamma_{i-1},\gamma_i).
\end{align*}
\end{definition}

\textbf{Remark.}  Note that we admit length functions $\ell$ with  $\ell (x,y) =0$ even if $b(x,y)>0$. This will simplify some considerations later on in the proof of Lemma \ref{l:boundedness_sigma_trees}. Let us emphasize, however, that these considerations would (with some additional effort) also work if we were to allow only length functions with $\ell (x,y) >0$ whenever $b(x,y) >0$.

\bigskip

Any length function $\ell$ on a connected  graph $(b,c)$ gives rise to a pseudometric $d_\ell$ defined  via
\begin{align*}
d_\ell  (x,y):=\inf\{\ell(\gm) : \gm \mbox{ is a path connecting $x$ and $y$}\}.
\end{align*}

\textbf{Remark.}   The assumption of connectedness of the graph forces that $d_\ell(x,y) < \infty$ for all $x, y \in X$.  Likewise, $\varrho(x,y) < \infty$ in this case, as follows from Lemma \ref{l:DLip} below. If the graph is not connected, the above definitions will still give pseudometrics on each connected component.   In the case that $c\equiv0$, $\varrho(x,y)=\infty$ for $x$ and $y$ in different connected components.   Likewise, we can extend the definition of $d_\ell$ to a function with values in $[0,\infty]$ by setting it equal to $\infty$ on any two points which do not belong to the same connected component.  \\ 

\smallskip

A natural choice of length functions is $\ell (x,y) = \frac{1}{b(x,y)^s}$ with a fixed parameter $s>0$.  The associated pseudometric is actually a metric (due to the summability assumption on $b$) and  will be denoted by $d_s$.
For $s=1$, we just write  $d$ instead of $d_1$.   Thus, we have
$$d (x,y) = \inf\{ \sum_{i=1}^n \frac{1}{b (x_{i-1}, x_i)} :  (x_0,\ldots, x_n) \;\:\mbox{is a path from $x$ to $y$}\}. $$
Let us remark that it is quite  natural to consider the inverse of $b(x,y)$ as a length since $b(x,y)$ represents the strength of the interaction between two vertices. In this context, we also  note that if  edges are replaced by real intervals of that length, then Brownian motion on the resulting space corresponds to the Dirichlet form $\ow Q$ given above (see \cite{GK} for details).

By a well-known inequality,  we  infer
$$ d^s \leq d_s$$
for all $0 < s \leq 1$.

\medskip

\textbf{Remark.} \label{rem-dm} Given a measure $m:X\to (0,\infty)$,  one obtains  another natural    length function,  namely,
$$\ell_{m}  (x,y):= \frac{\sqrt{m(x)m(y)}}{b(x,y)}. $$
Lengths of this type appear naturally when one deals with the unitary transformation $f\mapsto m^{1/2} f$ mapping the space $\ow D \cap \ell^2(X,m)$ to the space $D(Q_m) = \ow D_m \cap \ell^2(X,1)$.  Here, $\ow D_m$ are the functions of finite energy with respect to $Q_m$ where $Q_m$ is defined by
\[Q_m(f)= \frac{1}{2} \sum_{x,y\in X} b_m(x,y) |f(x)- f(y)|^2 + \sum_{x\in X} c_m(x) |f(x)|^2\]
with
\[b_m(x,y)= \frac{b(x,y)}{\sqrt{m(x)m(y)}}\]
and
\[c_m(x)= \frac{c(x)}{m(x)}+ \frac{1}{m(x)} \sum_{y\in X} b(x,y) - \sum_{y\in X} \frac{b(x,y)}{\sqrt{m(x)m(y)}}.\]
For this unitary transformation to be well defined, we need to assume that $\sum_{y\in X} \frac{b(x,y)}{\sqrt{m(y)}} < \infty$ for all $x\in X$. Note that, in this procedure, the potential $c_m$ will not necessarily stay nonnegative and thus, in a certain sense, we leave the setting of discrete Dirichlet spaces we consider.

\medskip

The following lemma shows that functions in $\ow D$ are $\frac{1}{2}$-H\"older continuous with respect to $d$.

\begin{lemma}\label{l:DLip} Let $ (b,c)$ be a connected graph over $X$. Then, for any  $f\in \ow D$ and all  $x,y\in X$, the inequality
\begin{align*}
|f(x)-f(y)|^{2}  \leq \ow Q(f) d(x,y)
\end{align*}
holds.
In particular, $\varrho^{2} \leq d$ and $\varrho \leq d_{1/2}$.  If, in addition, $c(x),c(y)>0$, then
\begin{align*}
|f(x)-f(y)|^{2}  \leq \ow Q(f) (c(x)^{-1} + c(y)^{-1})
\end{align*}
\end{lemma}
\begin{proof} 
Take a path $\gm=(\gm_{0},\ldots,\gm_{n})$ from $x$ to $y$. We estimate using the triangle inequality and the Cauchy-Schwarz inequality
\begin{align*}
|f(x)-f(y)|&\leq\sum_{i=1}^{n}b(\gm_{i-1},\gm_{i})^{1/2} |f(\gm_{i-1})-f(\gm_{i})|
\frac{1}{b(\gm_{i-1},\gm_{i})^{ 1/2 }} \\
&\leq\left(\sum_{i=1}^{n}b(\gm_{i-1},\gm_{i}) |f(\gm_{i-1})-f(\gm_{i})|^{2}\right)^{\frac{1}{2}}
\left(\sum_{i=1}^{n} \frac{1}{b(\gm_{i-1},\gm_{i})}\right)^{\frac{1}{2}} \\
&\leq \ow Q(f)^{\frac{1}{2}}l(\gm)^{\frac{1}{2}}.
\end{align*}
This implies the first  statement of the lemma. The `in particular' statement is then clear from the definition of $\varrho$ and the already mentioned inequality  $d^s \leq d_s$. The second inequality follows directly from
\[|f(x) - f(y)|^2 \leq (c(x)^{-1} + c(y)^{-1}) (c(x)|f(x)|^2 + c(y)|f(y)|^2).\]
This finishes the proof.
\end{proof}

\medskip

\textbf{Remark.} The second inequality of the previous lemma shows that the distance $\varrho$ might become much smaller than the distance $d$ if the killing term $c$  is nonzero. Moreover,  a positive  killing term
gives rise to another length function defined by
\[l_c(x,y)= c(x)^{-1} + c(y)^{-1}\]
if $x$ and $y$ are neighbors.  We will show in Appendix \ref{Reducing} that one can construct a metric where both $b$ and $c$ are involved by introducing a virtual point at infinity.

\medskip

Lemma~\ref{l:DLip} has various consequences which we will now discuss. We start by considering completions of $X$. We recall that the completion of $X$ with respect to a metric $\delta$ is denoted by $\overline{X}^{\delta}$.

\begin{theorem} \label{extension-d-to-varrho} Let $(b,c)$ be a connected  graph over $X$. Then, there exists a unique continuous map $\iota $  from $\overline{X}^d$ to $\overline{X}^\varrho$ extending the identity $X\longrightarrow X, x\mapsto x$. If $\overline{X}^d$ is compact, then $\iota$ is onto and $\overline{X}^\varrho$ is compact as well.
 Any $f \in \ow D$ has unique continuous extensions $f^{(d)}$ and $f^{(\varrho)}$ to $\overline{X}^d$ and $\overline{X}^\varrho$ respectively and $f^{(d)} = f^{(\varrho)} \circ \iota$ holds.
\end{theorem}
\begin{proof} Uniqueness of such a map $\iota$ is clear. Existence follows  immediately from the inequality $\varrho^2 \leq d$ proven in the previous lemma.

\smallskip

If $\overline{X}^d$ is compact, then its image under $\iota$ is compact as well. Hence, this image is complete and  contains, by construction, the set $X$. Thus, it must contain $\overline{X}^\varrho$. This shows surjectivity.

\smallskip

 By construction, all functions in $\ow D$ are Lipschitz continuous with respect to $\varrho$. Hence, they can be (uniquely) extended to continuous functions on $\overline{X}^{\varrho}$. Moreover, by the previous lemma, all functions in $\ow D$ are $1/2$-H\"older continuous with respect to $d$. Hence, they can be (uniquely) extended to $\overline{X}^d$.

 By continuity of $\iota$, for any $f\in \ow D$ the function  $f^{(\varrho)} \circ \iota$ is a continuous function on  $\overline{X}^{d}$ which agrees with $f$ on $X$. By the discussed uniqueness properties it must then agree with $f^{(d)}$ and
 the last statement of the theorem follows.
\end{proof}

Let us mention that, in general, $\iota$ is not an embedding, i.e., it is not injective as can be seen from Example~\ref{e:iota_no_embedding} in Section~\ref{counter}.

\smallskip

We now collect some basic facts concerning the Yamasaki space (cf. \cite[Lemma (3.14) and Theorem (3.15)]{Soa} or \cite[(2.4) Lemma]{Woe} for the case when $c \equiv 0$).

 We start with a consequence of Lemma~\ref{l:DLip}, namely, continuity of  point evaluation  for functions in $\ow D$.

\begin{lemma}\label{l:pointevaluation} If $(b,c)$ is connected, then the point evaluation map
\begin{align*}
    \de_{x}:(\ow D,\aV{\cdot}_{o})\longrightarrow \C, \quad u\mapsto u(x),
\end{align*}
is continuous for each $x\in X$.
\end{lemma}
\begin{proof} Let $x\in X$ and $f,g\in \ow D$. Then, we estimate by the previous lemma
\begin{align*}
|\de_{x} f-\de_{x} g|&\leq |(f(x)-g(x))-(f(o)-g(o))| +|f(o)-g(o)| \\ &\leq\ow  Q(f-g)^{\frac{1}{2}}d(x,o)^{\frac{1}{2}}+|f(o)-g(o)|.
\end{align*}
This implies the statement of the lemma as $d(x,o)<\infty$ for all $x$ whenever $b$ is connected.
\end{proof}

The previous important fact yields the following two well-known statements.

\begin{proposition}\label{p:Yamasaki} If $(b,c)$ is a connected graph over $X$, then the Yamasaki space $(\ow D,\as{\cdot,\cdot}_{o})$ is a Hilbert space.
\end{proposition}
\begin{proof} If $(f_{n})$ is a Cauchy sequence in $(\ow D, \as{\cdot, \cdot}_{o})$, then $(\de_{x}f_{n})$ converges to some $f(x)\in \C$ for each $x\in X$ by the lemma above.  Hence, $(f_{n})$ converges pointwise to a function $f$.  By Fatou's lemma and the fact that $(f_{n})$ is a Cauchy sequence we get
\begin{align*}
\ow Q(f)\leq\liminf_{n\longrightarrow\infty}\ow Q(f_{n})<\infty.
\end{align*}
As $(f_n)$ is a Cauchy sequence, another application of Fatou's lemma together with the pointwise convergence easily yields $\aV{f-f_{n}}_{o} \to0$ as $n\longrightarrow \infty$. \end{proof}

\begin{proposition}\label{l:continuity} Let  $(b,c)$ be a connected graph over $X$. Let $(B, \aV{\cdot})$ be a Banach space and $A : B \longrightarrow C(X)$ be a linear operator such that  $A (B) \subseteq \ow D$  and the map $B \longrightarrow C (X), u\mapsto A u,$ is continuous. Then, $A:B\longrightarrow C(X)$ is a continuous operator.
\end{proposition}
\begin{proof} By the closed graph theorem, it suffices to show that $A$ is a closed operator. As $A$ is defined everywhere, it suffices to show that it is closable. Thus, we have to show that $v \equiv0$  whenever $(u_n)$ converges to $0$ in $B$ and $(A u_n)$ converges to $v\in \ow D$. This follows easily from the continuity property of $A$ and the fact that   point evaluation is continuous in $\ow D$.
\end{proof}


\subsection{Intrinsic metrics}
Intrinsic metrics play an important role in the analysis of Laplacians on manifolds and  --- more generally --- of strongly local Dirichlet forms \cite{Stu}. For graphs (or general regular Dirichlet forms) this concept has only recently begun to be explored. In fact, they were only brought forward and  first treated systematically in \cite{FLW}. A fundamental idea of  \cite{FLW} is that  a metric   is intrinsic if and only if the associated set of Lipschitz functions belongs locally  to the form domain with `gradients' bounded by one  (in a suitable sense). In our context, this translates to the definition given below.
Further studies  involving intrinsic metrics on graphs can now  be found in various references, including \cite{BHK,BKW, HKW, HKMW}. We refer to these works for further discussion and  references.

\bigskip

In this subsection, we are given a graph $(b,c)$ over $X$ and a measure $m$ on $X$.

\smallskip

\begin{definition} Let $(b,c)$ be a graph over $(X,m)$. Then, a (pseudo)metric $\sigma : X\times X\longrightarrow [0, \infty)$ is called \emph{intrinsic} if for any $x\in X$ the inequality
$$ \frac{1}{2} \sum_{y\in X} b(x,y) \sigma^2 (x,y) \leq m(x)$$
holds.
\end{definition}

\textbf{Remark.} Of course, the  factor $1/2$ in the previous definition does not play a particular role and could  be replaced by any positive number. We use it as it simplifies some of the  formulae given  below. Also, with this factor, we are  completely in line with \cite{FLW} as well  as with  various subsequent works. Note, however, that some authors choose this factor to be one.

\medskip

It can easily be seen by examples that $\sigma$ can be unbounded even for a finite measure. At the end of this section, in  Corollary~\ref{c:bound-intrinsic}, we give a  geometric criterium which guarantees the boundedness of $\sigma$.
Also, observe that, even if an intrinsic metric $\sigma$ was allowed to take the value infinity, it would not do so on a connected graph.  This follows by the definition of intrinsic and the triangle inequality, see Lemma \ref{l:bound_sigma-via-d} below.

\smallskip

A consequence of the above definition is that there is a strong connection between intrinsic (pseudo)metrics  with respect to finite measures and functions in $\ow D$. This is discussed in the subsequent two propositions.

\smallskip

The first proposition, which follows immediately from definitions, shows that any function in $\ow D$ gives rise to an intrinsic pseudometric with respect to a measure of total finite mass.

\begin{proposition}\label{sigma-via-f} Let $(b,c)$ be a graph over $X$. Then, for any $f\in \ow D$ the function
$$\sigma_f : X\times X\longrightarrow [0,\infty), \; \sigma_f (x,y) = |f(x) - f(y)|$$
is an intrinsic pseudometric  with respect to the measure $M_f$ defined by
$$ M_f : X\longrightarrow [0,\infty), \: M_f (x) = \frac{1}{2} \sum_{y\in X} b(x,y) |f(x) -  f(y)|^2 + c(x) |f(x)|^2$$
and the  total mass of $X$ is given by $M_f (X) = \ow Q (f)$.
\end{proposition}

The next proposition gives a crucial ingredient in our considerations. It shows that intrinsic metrics belong to the domain of $\ow Q$ whenever the total mass of the space is finite.

\begin{proposition} Let $(b,c)$ be a graph over $(X,m)$  with  $m(X)<\infty.$   Let   $\sigma$ be   an intrinsic (pseudo)metric with respect to $m$  and $\ov{X}^\sigma$  the completion of $X$ with respect to $\sigma$. Let  $U\subseteq \ov{X}^\sigma$ be given and consider the  distance
$$\sigma_U : X\longrightarrow [0,\infty),\:\; \sigma_U (x) =\inf\{\sigma (x,a): a\in U\}.$$

 (a) If $c\equiv0$, then  $\sigma_U$ belongs to $\ow D$ and the estimate
$$\ow Q (\sigma_U) \leq   m(X)  $$
 holds.

 (b) If $C:= \sum_{x\in X} c(x) <\infty$ and $\sigma$ is bounded by $S\geq 0$, then $\sigma_U$ belongs to $\ow D$ and the estimate $\ow Q (\sigma_U) \leq  m(X) + C S^2$ holds.

\end{proposition}
\begin{proof} We only consider the case that $\sigma$ is a metric. The case where it is a pseudometric can be treated analogously.

\smallskip

(a)  By the triangle inequality and the definition of $\sigma_U$ we have, for any $x,y\in X$ with $b(x,y)>0$,
$$|\sigma_U (x) - \sigma_U (y) |\leq \sigma (x,y).$$
Now, a short calculation gives
\begin{eqnarray*}
\widetilde{Q} (\sigma_U) & = & \frac{1}{2} \sum_{x,y\in X}  b(x,y) |\sigma_U (x) - \sigma_U (y)|^2 \\
&\leq &  \frac{1}{2} \sum_{x,y\in X}  b(x,y) \sigma^2(x,y) \leq  \sum_{x\in X} m(x) =  m(X).
 \end{eqnarray*}
This finishes the proof  of part (a).

(b) This follows along very similar lines.
\end{proof}

\begin{proposition} \label{bound-sigma-via-varrho}  Let $(b,c)$ be a graph over $(X,m)$  with  $m(X)<\infty.$   Let   $\sigma$ be an intrinsic (pseudo)metric with respect to $m$. Assume $c \equiv 0$ or  $C := \sum_{x \in X} c(x) <\infty$ and boundedness of $\sigma$ if $c \not \equiv 0$. Then,  $$\sigma \leq A \varrho$$
with $ A = \sqrt{ m(X)}$ if $ c\equiv0$ and $A  = \sqrt{ m(X)  + CS^2 }$ with $S$ a bound on $\sigma$ and $C<\infty$ if $c\not \equiv 0$.
\end{proposition}
\begin{proof}
Let $x \in X$ be arbitrary  and consider the function  $f := \frac{1}{A} \sigma_{\{x\}} $ defining the distance to $x$ scaled by the factor $1/ A $.  Then, $f$ belongs to $\widetilde{D}$ and satisfies $\widetilde{Q} (f) \leq 1$ by the previous proposition. The definition of $\varrho$ then gives, for any $y\in X$, the inequality
$$\varrho(x,y) \geq |f(x) - f(y) |\geq \frac{1}{A} \sigma (x,y).$$
This proves the statement.
\end{proof}

The previous results allow one to establish strong relations between  $\varrho$ and intrinsic (pseudo)metrics. This is done in the next two theorems.

\begin{theorem}\label{rho-to-sigma} Let $(b,c)$ be a connected graph over $(X,m)$  with  $m(X) <\infty.$   Let   $\sigma$ be an intrinsic (pseudo)metric with respect to $m$. Assume $c\equiv0$ or $C:=\sum_{x \in X} c(x) <\infty$ and boundedness of $\sigma$ if $c \not \equiv 0$. Then, there exists a unique continuous map $$\gamma=\gamma_\sigma: \ov{X}^\varrho \longrightarrow \ov{X}^\sigma$$ extending the identity $X\longrightarrow X$. If $\ov{X}^\varrho$ is compact, then $\gamma$ is onto and $\ov{X}^\sigma$ is compact as well.
\end{theorem}
\begin{proof} Uniqueness of the map is clear by the denseness of $X$ in $\ov{X}^\varrho$. Existence  is a direct consequence of the previous proposition. Finally, the last statement on surjectivity and compactness of the range  follows by standard arguments as already used in the proof of Theorem~\ref{extension-d-to-varrho}.
\end{proof}

The previous theorem gives a relationship between $\varrho$ and one intrinsic metric. In fact, it is possible to compute  $\varrho$ via intrinsic (pseudo)metrics. This is done next.

\begin{theorem}\label{characterization-varrho-via-sigma} Let $(b,c)$ be a connected graph over $X$ with $c\equiv 0$.  Then,
$$\varrho = \sup\{ \sigma : \mbox{$\sigma$ intrinsic (pseudo)metric w.r.t. $m$ with $m (X) \leq  1$}\}.$$
\end{theorem}
\begin{proof}  Denote the function on the right hand side of the claimed equality by $\varrho^*$.

\smallskip

By Proposition \ref{bound-sigma-via-varrho}, we have  $\varrho \geq \sigma$ for any intrinsic (pseudo)metric with respect to  a measure $m$ with $m(X) \leq 1$. This gives $\varrho\geq \varrho^*$.

Conversely, by Proposition \ref{sigma-via-f}, for any $f\in \ow D$ with $\ow Q (f) \leq 1$ we obtain a  pseudometric $\sigma_f$ which is intrinsic with respect to a measure $M_f$ with $M_f (X) = \ow Q (f) \leq 1$ and
$$|f(x) - f(y) | = \sigma_f (x,y).$$
The definition of $\varrho$ then gives  $\varrho = \sup\{ \sigma_f : \ow Q (f) \leq 1\}$ and this implies $\varrho \leq \varrho^*$.
\end{proof}

\textbf{Remarks.} (a)   The theorem is in line with the ideas that  $\varrho$ is essentially the smallest metric making all functions in $\ow D$ Lipschitz continuous and that a metric is intrinsic if and only if all of its Lipschitz functions belong to the form domain with gradients uniformly bounded (in a certain sense).

(b)  In the proof of the theorem we used $\sigma_f$ for $f \in \ow D$. These are, in general,  not metrics but pseudometrics. For this reason, we have taken the supremum in the theorem  over pseudometrics.
  Let us note, however, that one could actually take the supremum over  metrics as well. The reason is that any pseudometric of the form $\sigma_f$ can be made into an intrinsic  metric by arbitrary small changes to $f$. This follows by a simple induction argument after one notices that for any $x\in X$ and $f \in \ow D$ the function   $$ \R \longrightarrow \R, \lambda \mapsto \ow Q (f + \lambda\delta_x) = \ow Q (f) + 2 \lambda \ow Q(\delta_x,f) + \lambda^2 \ow Q (\delta_x)$$
  is continuous.

(c) In general, $\varrho$ is not an intrinsic metric. This can be seen from Example \ref{e:C_does_not_imply_B}. Indeed, in this example the underlying graph is relatively compact with respect to any intrinsic metric but is not relatively compact with respect to $\varrho$. This also shows that, in general, the supremum above  is not a maximum.

\medskip

The following lemma provides a relationship  between intrinsic metrics and $d$.

As usual, the minimum of two real numbers $p$ and $q$ is denoted by
  $$p\wedge q :=\min\{p,q\}.$$

\begin{lemma}\label{l:bound_sigma-via-d}Let $(b,c)$ be a graph over $(X,m)$ and $\si$ an intrinsic metric. Then, for any path $(x_{0},\ldots,x_{n})$ connecting $x,y\in X$, we have
\begin{align*}
    \si(x,y) \leq \sqrt{2} \sum_{k=0}^{n-1}\left(\frac{m(x_k) \wedge m(x_{k+1})}{b(x_{k},x_{k+1})}\right)^\frac{1}{2}.
\end{align*}
In particular, we then have
$$\sigma (x,y)^2 \leq 2 \ m(\{x_0,\ldots, x_n\}) \sum_{k=0}^{n-1} \frac{1}{ b(x_k,x_{k+1}) }.$$
\end{lemma}
\begin{proof}
We compute directly, using the triangle inequality and  the intrinsic property,
\begin{eqnarray*}
    \si(x,y)&\leq &\sum_{j=0}^{n-1}\si(x_{j},x_{j+1})\\
     & = &  \sum_{j=0}^{n-1} \frac{1}{b(x_{j},x_{j+1})^{\frac{1}{2}}} (b(x_{j},x_{j+1})\si(x_{j},x_{j+1})^{2})^{\frac{1}{2}}\\
    &\leq & \sqrt{2} \sum_{j=0}^{n-1} \frac{\left(m(x_{j})\wedge m(x_{j+1})\right)^{\frac{1}{2}}}{b(x_{j},x_{j+1})^{\frac{1}{2}}}.
\end{eqnarray*}
This shows the first claim.  As for the second claim, we note that the Cauchy-Schwarz inequality directly gives
\begin{eqnarray*}
\sum_{k=0}^{n-1}\left(\frac{m(x_k) \wedge m(x_{k+1})}{b(x_{k},x_{k+1})}\right)^\frac{1}{2}&\leq& \left(\sum_{k=0}^{n-1} m(x_k)\wedge m(x_{k+1}) \right)^\frac{1}{2} \left( \sum_{k=0}^{n-1} \frac{1}{b(x_k,x_{k+1})}\right)^\frac{1}{2}.
\end{eqnarray*}
Combining this inequality with the first statement, one can now easily obtain the second claim by squaring.
\end{proof}

\textbf{Remark.} Let $\ell_m (x,y)= \sqrt{m(x)m(y)}/b(x,y)$ be the length function  introduced earlier and $d_{m,1/2}$ be the metric associated to the square root of this length function. Then,  the lemma  easily gives
\[\si(x,y) \leq \sqrt{2} \ d_{m,1/2}(x,y).\]
Indeed,  notice that $m(x_k)\wedge m(x_{k+1}) \leq \sqrt{m(x_k)m(x_{k+1})}$ and then take the infimum over all corresponding paths.

\medskip

 We will next discuss  various consequences of the previous lemma.  As a first consequence, we obtain  a condition for boundedness of intrinsic metrics on neighbors. In the case of trees even a converse is valid, see Section~\ref{counter}.

\begin{corollary}\label{l:boundedness_sigma} Let $(b,c)$ be a graph over $(X,m)$. If $$2a:= \inf_{x,y\in X, b(x,y)>0}\frac{b(x,y)}{m(x)\wedge m(y)}>0,$$
then every intrinsic metric is bounded  on neighbors by $a^{-\frac{1}{2}}$.
\end{corollary}
\begin{proof}
From Lemma~\ref{l:bound_sigma-via-d} we infer that
\[\si(x,y)^2 \leq 2 \ \frac{m(x) \wedge m(y)}{b(x,y)}.\]
Hence, the statement follows.
\end{proof}


\medskip

We also get the following immediate estimate. We remark that, unlike in  Proposition~\ref{bound-sigma-via-varrho}, we have no requirements on $c$ (as $c$ is neither  involved in $\si$ nor in  $d$).

\begin{corollary}\label{c:bound_sigma-via-d}Let $(b,c)$ be a graph and $\sigma$ an intrinsic metric for a finite measure $m$. Then, for all $x,y\in X$
\begin{align*}
    \si(x,y)^{2}\leq 2 \ m(X)d(x,y).
\end{align*}
\end{corollary}
\begin{proof} We obtain the statement by estimating $m(\{x_0,\ldots, x_n\}\leq m(X)$ in Lemma~\ref{l:bound_sigma-via-d}  and then taking the infimum over all paths. 
\end{proof}



As we see above, the assumption that an intrinsic metric $\sigma$ is bounded whenever $c$ does not vanish identically is important. Since for an intrinsic metric $\si$ every (pseudo)metric $\si'$ with $\si'\leq \si$ is obviously  intrinsic as well, there  always exist bounded intrinsic metrics.
From Corollary~\ref{c:bound_sigma-via-d} we now   obtain  directly the following criterion for boundedness of intrinsic metrics.

  \begin{corollary} \label{c:bound-intrinsic} Assume $(b,c)$ is such that $\diam \overline{X}^{d}<\infty$. Then, every intrinsic metric with respect to a finite measure is bounded.
  \end{corollary}

  The corollary can be substantially strengthened in the case  of locally finite graphs.

\begin{theorem}\label{diam-d-finite-implies-C} Let  $(b,c)$ be a locally finite connected  graph over $X$   such that $\diam \overline{X}^{d}<\infty$.  Then, $X$ is totally bounded with respect to any metric which is intrinsic with respect to a finite measure.
\end{theorem}
\begin{proof}  Let $m$ be a finite measure on $X$ and $\sigma$ be an intrinsic metric with respect to $m$. Chose $\varepsilon >0$ arbitrary. We have to show that $X$ can be covered by finitely many $\sigma$-balls with radius $\varepsilon$.

 Fix an arbitrary vertex $o\in X$. For $n\in \N$, denote by $B_n$ the set of all vertices in $X$ which can be reached from $o$ in not more than $n$-steps (i.e., those vertices $x\in X$ for which there exist $x_1,\ldots, x_k\in X$ with $k\leq n$, and $x_1 = o$, $x_k =x$ and $b(x_j,x_{j+1})>0$, $j=1,\ldots, k-1$). Then, each $B_n$ is finite (as the graph is locally finite). Moreover, the union of the $B_n$  equals $X$. As  $m (X)$ is finite, we can now chose  $n$ large enough such that
 $$ 4 \: m (X\setminus B_n) \:    \diam\overline{X}^{d} \leq \varepsilon^2.$$
 Set $N:= n+1$. Then, $B_{N}$ is finite.

 \smallskip

 \textit{Claim.} Any  point of $X$ has $\sigma$ distance to $B_{N}$ less than $\varepsilon$.

\emph{Proof of the claim}:
 If the point $p\in X$ belongs to $B_{N}$ this is clear. Otherwise, we can chose a path $p=x_1,\ldots, x_k = o$  from $p$ to $o$ such that  $k\geq N+1$ and
$$ \sum_{j=1}^{k-1} \frac{1}{b(x_j, x_{j+1})} < 2 \diam \overline{X}^{d}.$$
Let $s\in \{1,\ldots, k\}$ be the smallest  integer such that $x_s$ belongs to $B_N$. Then, $x_s$ does not belong to $B_{n}$ (otherwise $x_{s-1}$ would belong to $B_{n+1} = B_N$, which is a contradiction). Hence, $m (\{x_1, \ldots, x_s\}) \leq m (X \setminus B_n)$ holds. Invoking   Lemma \ref{l:bound_sigma-via-d}, we now obtain
\begin{eqnarray*}
\sigma (p,x_s)  & \leq &  \left( 2 m(\{x_1,\ldots, x_s\}) \sum_{j=1}^{s-1} \frac{1}{b(x_j, x_{j+1})}\right)^{1/2} \\
& \leq &  \left(  m(X\setminus B_n) \:  4 \:  \diam \overline{X}^{d}\right)^{1/2}\\
& < &  \varepsilon.
\end{eqnarray*}
Here, we used the definition of $n$ in the last line. This proves the claim.

\smallskip

Due to local finiteness, the set $B_N$ is finite. Thus, the statement of the theorem follows from the claim.
\end{proof}


\subsection{Effective resistance and $\varrho$ via  electrical networks} \label{seceffres}

In this subsection, we express the metric $\varrho$ defined above in the standard terminology of electrical networks as in textbooks like  \cite{JP2, LyonsBook}. In this context, we also highlight  Kigami's work \cite{Kig01,Kig}, which gives a treatment of Dirichlet forms on (metric)  graphs centered around the  resistance metric.

\bigskip

We start by introducing a quantity which can be related to the free effective resistance in the case of locally finite graphs with $c\equiv0$ (see Proposition \ref{rho-free} below). Let $(b,c)$ be a graph over $X$. Define $r : X\times X\longrightarrow [0,\infty)$ via
$$r(x,y)= \sup\{ 1/\ow Q(g) : g\in \ow D,  |g(x) - g(y)|=1\}$$
for $x\neq y$ and $r(x,x)=0$.

\begin{theorem} \label{t:rhofree} Let $(b,c)$ be a connected graph over $X$. Then,
$r= \varrho^2$.
\end{theorem}

\begin{proof}
Let $x,y\in X$.
Recall that $\varrho(x,y)$ is defined as
$$\varrho(x,y)=\sup\{|f(x)-f(y)| : f\in \ow D,\ow Q({f})\leq 1\}.$$
Let $\eps>0$ and $f\in \ow D$, $\ow Q({f})\leq1$, be a function such that $s:=|f(x)-f(y)| \ge \varrho(x,y)-\eps>0$. Then, the function $f/s$ is a candidate for $g$ in the definition of $r$ above.  We estimate
$$r(x,y)\geq \frac 1{\ow Q(f/s)}= \frac{s^{2}}{\ow Q(f)}\geq s^{2}\geq(\varrho(x,y)-\eps)^2.$$
As $\eps>0$  is arbitrary we infer that
$$ r(x,y) \geq \varrho(x,y)^2.$$

Conversely, let $\eps>0$ and $g\in \ow D$, $|g(x)-g(y)|= 1$, be a function such that $1/\ow Q({g})\geq r(x,y)-\eps\ge0$. Then, the function $g/\ow Q(g)^{\frac{1}{2}}$ satisfies $\ow Q(g/\ow Q(g)^{\frac{1}{2}} ) = 1$, and, so, it is a candidate for $f$ in the definition of $\varrho(x,y)$. This means that
$$\varrho(x,y)\geq \frac{|g(x)-g(y)|}{\ow Q(g)^{\frac{1}{2}}}= \frac{1}{\ow Q(g)^{\frac{1}{2}}}\ge ({r(x,y)-\eps})^{\frac{1}{2}}.$$
Combining these inequalities  we obtain  $r= \varrho^2$ as claimed.
\end{proof}

\textbf{Remark.}  As discussed in the first remark of Subsection \ref{Metrics}, the supremun in the definition of $\varrho (x,y)$  can be replaced by a maximum. Using this in the above proof, we can easily infer that the supremum can also be replaced by a maximum in the definition of $r$.

\medskip

We next sketch the connection to the free effective resistance metric. For a  graph $(b,0)$ over a finite set $X$, the map $r$ is called
 \emph{effective resistance}. It is shown in \cite[Theorem~2.1.14]{Kig01} and \cite[Lemma~4.5]{JP2} that $r$ is a metric  and, thus, it is called the \emph{effective resistance metric} in this case.

This definition is, in fact, just one of the many well-known equivalent definitions of effective resistance. The physical intuition behind this is that the electrical potential on $X$ induced by an external voltage source applied to $x$ and $y$ is the energy minimizer among all functions with the given boundary conditions; recall the formula $E= V^2/R_{\mathrm{eff}}$ from elementary physics, expressing the energy $E$ in terms of the imposed voltage $V$ and effective resistance $R_{\mathrm{eff}}$.

When $X$ is infinite, there are at least two standard notions of effective resistance; the one most relevant to this paper, which coincides with $r$ described above, is the free effective resistance defined as follows.

Given a finite subset $X'$ of $X$, we think of the restriction $b'$ of $b$ to $X'$ as a weighted graph and call it the \emph{subgraph} induced by $X'$.
If $(b,0) $ is additionally locally finite, we define the \emph{free effective resistance} between $x$ and $y$ as the infimum over the effective resistances between $x$ and $y$ for all finite subgraphs containing $x$ and $y$.  We denote this quantity by  $r_{\mathrm{free}}(x,y)$.

In this way, $r_{\mathrm{free}}$ is defined via an approximation procedure by finite graphs. It turns out that it can, in fact, be expressed by virtually the same term as $r$ in the case of finite graphs above. This is well-known. For example  \cite[Theorem 4.2]{JP2} combined with \cite[Theorem 4.12]{JP2} or \cite[Theorem 2.3.4]{Kig01}
(see \cite[Exercise~9.41]{LyonsBook}  as well) show the following.

\begin{proposition} \label{rho-free} Let $(b,0)$ be a locally finite connected graph over $X$. Then, $r_{\mathrm{free}} (x,x) =0$ and, for every pair $x,y$ of distinct vertices in $X$, the equality
$$r_{\mathrm{free}}(x,y)=\max\{ 1/\ow Q(g) : g\in \ow D,  |g(x) - g(y)|=1\}=r(x,y)$$
holds. Moreover,  the map
$r$ is a metric on  $X$.
\end{proposition}

\textbf{Remark.} The proofs  of the previous proposition presented in the literature do not seem to cover  a non-locally finite setting. However, for our applications in Section~\ref{counter}, this is no restriction as we consider only locally finite examples there.


\section{Topology: Canonically compactifiable graphs}\label{section-topology}

Let $X$ be a countable set, $m$ a measure on $X$, and $(b,c)$ a weighted graph over $(X,m)$. We give a characterization of when  $X$ is relatively compact in a natural way. This will be used in order to   define a  boundary  $\partial X$ of $X$. Our considerations are based on  $C^\ast$-algebra techniques.


\subsection{Basic definition and features} \label{Canonically}
In this subsection, we study the topology of a graph by assigning a natural $C^\ast$-algebra to it. In order to do so, we need $\ow D$ to be an algebra of bounded functions.

\begin{definition} The graph $(b,c)$  is \emph{\cpt} if $\ow D\subseteq \ell^{\infty}(X)$.
\end{definition}

\textbf{Remark.}
Let us point out that   both $b$ and $c$ can play a role in making a graph \cpt.  In fact, it is  not hard to construct examples of graphs $(b,c)$  such that $(b,0)$ is not {\cpt}  but $(b,c)$ is.  To construct such an example, it suffices to consider   $X=\N$ with $b(n,n') = 0$ for $|n-n'| \neq 1$ and $b(n,n+1)=1$  and $c(n)\longrightarrow \infty$ as $n\longrightarrow \infty$. Then, $(b,c)$ obviously is  \cpt. However, the function
$$ f : \N\longrightarrow \R, \:\; f(n) = \sum_{j=1}^n \frac{1}{j},$$
is unbounded and has finite energy on $(b,0)$.

\medskip

The following lemma shows that the embedding $\ow D\subseteq\ell^{\infty}(X)$ is actually continuous.

\begin{lemma}\label{l:cpt} If $(b,c)$ is \cpt\ and connected, then the embedding
$$j:(\ow D,\aV{\cdot}_{o}) \longrightarrow \ell^{\infty}(X)$$
is continuous.
\end{lemma}
\begin{proof}
Assume that there is a sequence $(f_{n})$ in $\ow D$ and a function $g\in\ow D$ such that $f_{n}\longrightarrow 0$ in $\aV{\cdot}_{o}$ and $j(f_{n})\longrightarrow g$ in $\aV{\cdot}_{\infty}$ as $n\longrightarrow\infty$.
If we show that $g\equiv0$, then the statement follows from the closed graph theorem.  Since, by Lemma \ref{l:pointevaluation}, point evaluation is continuous whenever $b$ is connected, we have that $j(f_{n})(x)=f_{n}(x)\longrightarrow 0$ for all $x\in X$. As  $j(f_{n})\longrightarrow g$ with respect to $\aV{\cdot}_{\infty}$ we infer  $g\equiv0$.
\end{proof}

We next  give a characterization as well as  sufficient conditions   for compactifiability in terms of the metrics $\varrho$ and $d$  and the intrinsic metrics introduced in the previous section.

Recall that $\diam_\sigma (X) :=\sup_{x,y\in X}\sigma(x,y)$ denotes the diameter of $X$ with respect to any {(pseudo)metric} $\sigma$.

\begin{theorem}\label{characterization-cpt} Let $(b,c)$ be a  connected graph over $X$. Then the following assertions are equivalent:

\begin{itemize}

\item[(i)] $(b,c)$ is {\cpt}.

\item[(ii)]  $\diam_{\varrho}(X)<\infty$.

\item[(iii)] $\diam_{r} (X) < \infty$.

\end{itemize}

If $c\equiv0$, this is furthermore equivalent to the  following assertion:
\begin{itemize}
\item[(iv)] $\diam_\sigma (X) < \infty$ for any  pseudometric $\sigma$ intrinsic  with respect to a finite measure $m$ on $X$.

\end{itemize}

\end{theorem}
\begin{proof}  (ii)$\Longrightarrow$(i):  Set  $C:=\diam_{\varrho}(X)<\infty$ and let $o\in X$.  By the definition of $\varrho$, we then obtain $|f(o)-f(x)|\leq C$  for all $f\in \ow D$ with $\ow Q(f)\leq1$ and all  $x\in X$. Hence, $f\in\ell^{\infty}(X)$ for $f\in \ow D$ with $\ow Q(f)\leq1$ which implies that $\ow D\subseteq \ell^{\infty}(X)$.

\smallskip

(i)$\Longrightarrow$(ii):  By Lemma~\ref{l:cpt}, we have that $\aV{f}_{\infty}\leq || j ||$ for every $f\in\ow D$ with $\aV{f}_{o}\leq1$. Hence, for all $x,y\in X$,
$$\diam_{\varrho_{o}}(X) =\sup_{x,y\in X}\sup_{\aV{f}_{o}\leq 1}|f(x)-f(y)|\leq 2 \sup_{\aV{f}_{o}\leq 1}\aV{f}_{\infty}=2 \ || j ||$$
By Proposition~\ref{p:rho}, we conclude that $\diam_{\varrho}(X)<\infty$.

\smallskip

\smallskip
\smallskip

(ii)$\Longleftrightarrow$(iii): This is immediate from Theorem \ref{t:rhofree}.

Now, assume that $c\equiv0$.

\smallskip

(iv)$\Longrightarrow$(i): Let $f\in \ow D$ be arbitrary. Then, $ \sigma := \sigma_f$ is an intrinsic pseudometric with respect to a finite measure  by Proposition \ref{sigma-via-f}. Fix now an arbitrary $o\in X$. Then,
$$ |f(x)| \leq |f(x) - f(o)| + |f(o)| = \sigma (x,y) + |f(o)| \leq \diam_\sigma (X) + f(o)$$
holds for any  $x\in X$. As $\diam_\sigma (X)$ is finite by (iii), we infer that $f$ is bounded.

\smallskip

(ii)$\Longrightarrow$(iv): This is immediate from Proposition \ref{bound-sigma-via-varrho}.
\end{proof}

\textbf{Remark.} Due to Theorem \ref{characterization-varrho-via-sigma}, condition (iv) in the previous Theorem is equivalent to the (a priori stronger) condition:

\begin{itemize}

\item[(iv')]  There exists a $\beta\geq 0$ such that $\diam_\sigma (X) \leq \beta$ for  every metric $\sigma$ which is intrinsic with respect to  a measure $m$ with $m (X) \leq 1$.
\end{itemize}

\medskip

The following corollaries give  sufficient condition for compactifiability via conditions on the  metric $d$ introduced in Section~\ref{Metrics} and  the intrinsic metrics. Let us mention that, in general, the converse of Corollary \ref{sufficient-condition-cpt-d} does not hold, see Example~\ref{e:diam_d_and_rho} in Section~\ref{counter}.

\begin{corollary}\label{sufficient-condition-cpt-d} If $\diam_{d}(X)<\infty$, then $(b,c)$ is \cpt.
\end{corollary}
\begin{proof}
We have  $\varrho^{2}\leq d$ by Lemma~\ref{l:DLip}. Therefore, the statement follows from the previous theorem.
\end{proof}

\begin{corollary}\label{sufficient-condition-cpt-intrinsic} Let $(b,c)$ be a connected graph over $X$ with $c\equiv0$  such that $\ov{X}^\sigma$ is compact for any  pseudometric $\sigma$ which is intrinsic with respect to a finite measure. Then $X$ is \cpt.
\end{corollary}
\begin{proof} Compactness of $\ov{X}^\sigma$ obviously implies finiteness of the diameter. Thus, the corollary follows from the previous theorem.
\end{proof}

We next give two  classes of examples of graphs where the assumptions of Corollary \ref{sufficient-condition-cpt-d} are clearly satisfied.

\begin{example} (Summable $1/b$)  Let $(b,c)$ be a graph over $X$ such that
$$\sum_{x,y\in X, b(x,y)> 0}\frac{1}{b(x,y)}<\infty.$$ Clearly, in this case, $\diam_{d}(X)<\infty$ and $(b,c)$ is {\cpt} by Corollary \ref{sufficient-condition-cpt-d} (in Subection \ref{key-example} we will have a closer look at this case).
\end{example}

 \bigskip

The condition of summability of $1/b$ is by no means necessary in order to obtain a {\cpt} graph. This can easily be seen by considering trees. A first glimpse at this class of examples will be given next. A further study of this class will be undertaken in Section~\ref{counter}.

\begin{example}
Let $(b,c)$ be a tree, i.e., a connected graph without cycles. Then, between any two points, there exists a unique path and the length of this  path defines the value of the metric $d$ between the points. Fix now an arbitrary vertex  $o\in X$. Then,  $\diam_{d}(X)<\infty$  holds if and only if there is a $C>0$ such that for every natural number $n$ and any  path $\gm = (\gm_{0},\ldots,\gm_{n})$  starting at $o$ we have $\sum_{i=1}^{n}\frac{1}{b(\gm_{i-1},\gm_{i})}\leq C$. In particular, every such path is bounded.
\end{example}






\subsection{The compactification $K$} \label{secBoCo}
Any  graph which is {\cpt} comes naturally with a canonical compactification $K$. This is discussed in this subsection. In the next subsection, we will see that this compactification $K$ actually agrees with the Royden compactification (Theorem~\ref{KR}).

 We start with a lemma which is certainly well-known.

\begin{lemma}\label{l:algebra}
If $(b,c)$ is \cpt, then  $\ow D$ is an algebra and
 $$\A:=\mbox{Closure of $\ow D$ with respect to $\aV{\cdot}_\infty$}$$
is a commutative $C^\ast$-algebra.
\end{lemma}
\begin{proof}
Notice that for $f,g\in \ow D\cap \ell^{\infty }(X)$
\begin{align*}
\ow Q(fg)&\leq 2\aV{g}_{\infty}^2 \ow Q(f) + 2\aV{f}_{\infty}^2 \ow Q(g)
\end{align*}
by simple algebraic manipulations (compare Lemma \ref{estimate-product} in Appendix \ref{A:Royden}).
It follows that $\ow D\cap \ell^{\infty }(X)$ is an algebra and, by assumption, $\ow D=\ow D\cap \ell^{\infty}(X)$.
Clearly, with complex conjugation as involution, $\ow D$ is a commutative involutive algebra that satisfies $\aV{|f|^{2}}_{\infty}=\aV{f}_{\infty}^{2}$ for $f\in\ow D$. Hence, the closure of $\ow D$ with respect to $\aV{\cdot}_{\infty}$ is a $C^\ast$ algebra.
\end{proof}

\textbf{Remark.} Under the additional assumption that $c\equiv0$, it is shown in \cite{Soa}  that the space of bounded functions in $\ow D$ with the norm $ \ow Q(\cdot)^{\frac{1}{2}} + \aV{\cdot}_{\infty}$ forms a commutative Banach algebra, called there the \emph{Dirichlet algebra} \cite[Theorem (6.2)]{Soa} (see further discussion in Appendix \ref{A:Royden} as well). \\


For a {\cpt} graph $(b,c)$ over $X$ we define
 $$\A^{+}:=\mbox{smallest $C^\ast$-algebra containing $1$ and $\A$}.$$
Then, $\A^{+}$ is a  commutative unital $C^\ast$-algebra. It will be the key object in our subsequent study in this subsection.

Note that  it can occur that $\A$ already contains $1$. In this case, $\A$ is  unital and agrees with   $\A^{+}$. This case can be characterized as follows:

\begin{proposition}\label{p:one} Assume that $(b,c)$ is \cpt. Then, the following assertions  are equivalent:
\begin{itemize}
  \item [(i)] $1\in\A$
  \item [(ii)] $1\in \ow D$
  \item [(iii)] $\sum_{x}c(x)<\infty$.
\end{itemize}
\end{proposition}
\begin{proof}
Clearly (iii)$\Longrightarrow$(ii)$\Longrightarrow$(i). Assuming (i), we conclude that there are $f_{n}\in \ow D$ such that $f_{n}\longrightarrow 1$ in $\ell^{\infty}(X)$. Hence, there is an $N$ such that $|f_{N}(x)|\geq 1/2$ for all $x\in X$. Since $f_{N}\in\ow D$, we have that $\tfrac{1}{4}\sum_{x}c(x)\leq\sum_{x}c(x)|f_{N}(x)|^{2} \leq \ow Q (f_N) <\infty$,  which implies (iii).
\end{proof}

\textbf{Remarks.} (a)  Of course (ii)$\Longleftrightarrow$(iii) also holds if $(b,c)$ is not \cpt.

(b) The proposition gives one reason for the relevance of graphs $(b,c)$  with summable $c$.

\smallskip

Whenever $(b,c)$ is a {\cpt} graph over $X$, we denote the set of characters of  $\A^+$ by $K$ and the set of characters of $\A$ by $K'$. (Here, the \textit{characters} are the  non-trivial, linear  multiplicative maps from the $C^\ast$-algebra to the complex numbers.) Of course, $K = K'$ holds if $\A = \A^+$.

By the commutative Gelfand Naimark theorem, $K'$ is  a locally compact Hausdorff space and  $\A$ is isometrically isomorphic to $C_0 (K')$ (where $C_0 (K')$ is the space of continuous functions vanishing at infinity on $K'$ which is understood to agree with $C (K')$ if $K'$ is compact). More precisely, the situation is as follows:  The set $K$ is   compact with respect to the weak-*-topology and there is an isometric isomorphism, known as the \textit{Gelfand map},
\begin{align*}
    \A^{+}\longrightarrow C(K),\quad g\mapsto\oh g, \:\;\mbox{with $\oh g (\gamma) = \gamma (g)$}.
\end{align*}
Moreover, if  $1\not\in\ow D$, then $K'$ is not compact and $K$ is its one-point compactification. In fact, in this case, the compact $K$ satisfies
$$K = K'\cup\{\gamma_\infty\}, $$
 where
$$\gm_{\infty}(g+\al1)=\al,\qquad g\in \A,
\al\in\C.$$
In  this case,   the  Gelfand map  induces an  isometric isomorphism
$$\A\longrightarrow \{f\in C(K) : f(\gm_{\infty})=0\}= C_{0}(K\setminus\{\gm_{\infty}\}) = C_{0} (K'), \; f \mapsto \oh f|_{K'}.$$

\smallskip

With slight abuse of notation, we will denote  the supremum norm on both $\A$ and $C(K)$ with  $\aV{\cdot}_{\infty}$.

\smallskip

The next theorem shows that $K$ is a compactification of $X$ in the usual sense.

\begin{theorem}\label{t:open} Let $j:X\longrightarrow K, x\mapsto \de_{x}$ with  $\de_{x}:\A^+ \longrightarrow\C$, $f\mapsto f(x)$ be the canonical embedding. Then,  $j(X)$ is a dense open subset of  $K$. In fact, for every $x\in X$, the set $j(\{x\})$ is open in $K$.
\end{theorem}
\begin{proof} Clearly, $\de_{x}$ for $x\in X$, is a character of $\A^+$ and, thus, $j$ indeed maps $X$ to $K$ (and even to  $K'$). We will show  denseness of $j(X)$ in $K$ as well as openness of the sets  $\{j(x)\}$ for any $x\in X$.

We restrict to the case that $1$ does not belong to $\ow D$. The other case is simpler and can be treated similarly.

We first show that $X$ is dense in $K$. Choose $k\in K $, $k\neq \gm_{\infty}$, arbitrarily and let $W\subseteq K$ be an open neighborhood of $k$. We show that there is an $x\in X$ such that $j(x)\in W$.
By Urysohn's lemma, there is a function $\ph\in C(K)$ with $\ph(k)=1$, $\ph\equiv 0$ on $K\setminus W$ and $0\leq \ph\leq 1$. As $\ow D$ is dense in $\A$, we can choose $g\in\ow D$ with $\aV{\oh g-\ph}_{\infty}\leq 1/3$. Consider $W_{k}=\{\gm\in K : |\oh g(\gm)|>1/2\}$. Then,  $W_{k}$  is open, $W_{k}\subseteq W$ and $k\in W_{k}$.  Moreover, $1-\aV{\oh g}_{\infty}\leq \aV{\ph-\oh g}_{\infty}\leq 1/3$. As $\aV{\oh g}_{\infty}=\aV{ g}_{\infty}$, this implies that $\aV{ g}_{\infty}\geq 2/3$. Therefore, there exists $x\in X$ such that $|g(x)|\geq2/3$. This implies, $|\oh g(\de_{x})|=|\de_{x} g|=|g(x)|\geq 2/3$ and thus $j(x)=\de_{x}\in W_{k}$.

 Now, let $k=\gm_{\infty}$. If $\gm_{\infty}$ is a discrete point, then $K\setminus \{\gm_{\infty}\}$ is compact. However, this implies that $C(K\setminus\{\gm_{\infty}\})$ is unital which is a contradiction.  Let $W$ be an open neighborhood of $\gm_{\infty}$. Since $\gm_{\infty}$ is not a discrete point of $K$ there exists  $k\in W$, $k\neq\gm_{\infty}$. By Urysohn's lemma there exists $\ph\in C(K)$ with $\ph(k)=1$, $\ph(\gm_{\infty})=0$, $0\leq \ph\leq 1$ and $\ph\equiv 0$ on $K\setminus W$. 
We continue as above to show that there exists $x\in X$ with $j(x)\in W$.

Hence, we have shown that for every $k\in K$ and every neighborhood of $k$ there exists $x\in X$ such that $j(x)$ is in this neighborhood.  Therefore, $j(X)$ is dense in $K$.

\medskip

We are now going to show that the sets $\{j(x)\}=\{\de_{x}\}$ are open for all $x\in X$. It suffices to show that there are $\oh f:K\longrightarrow \C$ continuous and $A\subseteq\C$ open such that $\{\de_{x}\}=\oh f^{-1}(A)$.
Let $f=1_{x} \in \ow D$ be given and consider  and $A=\{z\in \C : |z|>1/2\}$.
Then, $f$ takes only the values $0$ and $1$ and this must then be  true for $\oh f$ as well. Clearly, $\oh f$ assumes the value $1$ on $\delta_x$. This immediately gives  $\{\de_{x}\}\subseteq \oh f^{-1}(A)$.

 Consider now an arbitrary $\gm\neq\de_{x}$. Since $j(X)$ is dense in $K$, there exist $\de_{y_n}$ with $y_{n}\in X$ such that $\de_{y_{n}}\longrightarrow \gm$, $n\longrightarrow\infty$. As $\gamma \neq \delta_x$ we can furthermore assume without loss of generality that  $\delta_{y_n} \neq \delta_x$ for all $n$. This gives
$$\oh f (\gamma) = \lim_{n\longrightarrow \infty} \oh f (\delta_{y_n}) = \lim_{n\longrightarrow \infty} f(y_n) = 0.$$
 Hence, $\gamma$ does not belong to $\oh f^{-1} (A)$. Therefore,  $\{\de_{x}\}= \oh f^{-1}(A)$.  As  $\{j(x)\}$ are open for every $x\in X$, the set  $j(X)$ is a open subset of $K$ as well.

\medskip

Now, all of the statements of the theorem are proven.
\end{proof}

\textbf{Remark.} Consider a {\cpt} graph $(b,0)$ over $X$. Then, it is not hard to see that the associated algebra $\A = \A^+$ does not change if $(b,0)$ is replaced by $(b,c)$ with summable $c$. Thus, $K$, the character space of this algebra, also does not change. In this sense the compactification $K$ is stable under the addition of a summable $c$. This is not surprising since, as observed in  Section~\ref{secWG} and further described in Appendix \ref{Reducing}, the role of $c$ can be simulated by one  imaginary vertex at infinity and this vertex will not change the compactification.

\medskip

The previous theorem shows that we can view $X$ as an open subset of $K$. With slight abuse of notation we will not distinguish between $X$ and $j(X)$ to view $X$ as a subset of $K$.


\subsection{$K$ is the Royden compactification $R$}
It turns out that the compactification we have constructed in the previous subsection  agrees with the well-known Royden compactification $R$.  This is shown in this subsection.  In order to be more specific, we first must clarify the role of the killing term  $c$ in our considerations. The Royden compactification is commonly only considered in the case $c\equiv0$, i.e., for graphs of the form $(b,0)$ over $X$. However, it turns out that the standard construction of this compactification as given, e.g., in \cite{Soa} can be extended to arbitrary graphs $(b,c)$. For the convenience of the reader, we discuss the details  of the corresponding  construction in Appendix \ref{A:Royden}.
The outcome of this construction is what we call the Royden compactification in this section. It is this compactification that we show to be equal to $K$.


\begin{theorem} \label{KR}
Let $(b,c)$ be \cpt.  Then $K$ is homeomorphic to the Royden compactification $R$.
\end{theorem}
\begin{proof} We will use the considerations of Appendix \ref{A:Royden}. Note that the algebra $\mathcal{B}$ considered there agrees with $\ow D$ as our graph is {\cpt}.  By Theorem \ref{Characterization-Y}, there is then  a unique (up to homeomorphism)  locally compact Hausdorff space $Y$ such that the following holds:
\begin{enumerate}
\item The set $X$ is  a dense  open subset of  $Y$.

\item Any function in $\ow D$  can be extended to an element of $C_0 (Y)$.

\item The algebra $\ow D$ separates points of $Y$.

\item The algebra $\ow D$ does not vanish on any point of $Y$.

\end{enumerate}

We will show that the space $K'$ of characters of $\A$ satisfies these properties. This will then imply  the statement (as, in the case $1\in \ow D$, we have $K = K' = Y = R$ and, in the case $1\notin \ow D$, we have that  $K$ is the one-point compactification of $K'$ and $R$ is the one-point compactification of $Y$.)


Property (1) is  immediate  from Theorem~\ref{t:open}.

\smallskip

 To show  property  (2),  notice that any element of $\ow D$ belongs to the algebra $\A$. The algebra $\A$, in turn, is isometrically isomorphic to $C_0 (K')$ as discussed in the previous section. (Here, we use that $C_0 (K') = C (K')$ if $K'$ is compact.)
Similarly, property (3)  follows: By the definition  of $K'$,  the elements of $\A$ trivially  separate the points of $K'$. Now, by the construction of $\A$, we have that $\ow D$ is dense in $\A$ and hence the elements of $\ow D$ separate the points of $K'$ as well. Finally, property (4) follows as, by the very definition of $K'$ as non-trivial characters, the algebra $\A$ cannot vanish identically on any element of $K'$.
\end{proof}

Let us stop for a moment and give a discussion of the relationship between our approach leading to $K$ and earlier work along related lines concerning $R$:

The idea of embedding a graph in a compact space via Gelfand theory  is not new. It can already be found in the work of Yamasaki \cite{Yam} and Kayano and Yamasaki \cite{KY}.  A thorough discussion  is then undertaken in Chapter VI of \cite{Soa} (compare with Appendix \ref{A:Royden} of the present paper). There,  a graph $(b,c)$ over $X$ with $c\equiv0$ is considered and it is shown that the space
$${\ow D}\cap \ell^\infty (X)$$
is a Banach algebra when equipped with the norm
$$\|u\|_{\ow Q, \infty}:= \ow Q (u)^{1/2} + \|u\|_\infty.$$
Note that this algebra  always contains the constant functions (due to the assumption $c\equiv0$).  By the Gelfand theory, this Banach algebra gives rise to a compact space $R$ and this space is called the \textit{Royden compactification} in \cite{Soa} after the work of Royden on Riemannian surfaces \cite{Roy}.

Now, our approach differs substantially from this approach in a number of ways. Most importantly, the main point of our work is to \textit{single out a class of graphs} which can intrinsically be seen as relatively compact rather than to associate a compactification to any graph.

Also, on  a technical level, two differences are worth pointing out. One difference is that we work with $C^\ast$-algebras rather than Banach algebras. This involves taking an additional limit when going from $\ow D$ to $\A$. This limit  will be of crucial importance in our treatment of the Dirichlet Problem later on. In fact, it is exactly this further limit that will allow us to  solve it in a rather general context.
 The  other difference is that we are not restricted to the  case $c \equiv 0$. This is relevant as our perspective is to deal with arbitrary regular Dirichlet forms on discrete spaces (and these Dirichlet forms are in one-to-one correspondence with  graphs $(b,c)$ with arbitrary $c\geq 0$).  In fact, in order to state the previous theorem, we even had to extend the common definition of Royden compactification to the framework of general graphs $(b,c)$ over $X$.

 Given these differences between our approach to $K$ and the construction of the Royden compactification, the statement of the preceding theorem seems rather  remarkable in that it shows that one still obtains the same compactification.

To underline this last point we now include an example showing that, in general, $\ow D$ is strictly smaller than $\A$ in the case of \cpt\ graphs.

\begin{example}[$\ow D$ strictly contained in $\A$]\label{DnequalA}
We consider the one-way infinite path $X = \N$ with  weights
$$b(n,n+1) = n^3$$
(and $b(x,y) = 0$ for $|x-y|>1$)
and $c\equiv0$.  This  graph  satisfies
$$\sum_{n\in X} \frac{1}{b(n, n+1)}<\infty$$
and  is, therefore, \cpt\ (as shown in Example~\ref{key-example}).
Now, it can be seen directly that the function
$$f(n)=1/n,    \quad n\ge1,$$
is not of finite energy and, hence, does not belong to $\ow D$.  However, for any $k\in \N$,  the function
$$f_k(n) = \frac{1}{ n^{1+1/k} },  \quad     n\ge1,$$
 is of finite energy (as follows by estimating $ (f_k ( n+1)-f_k(n))$  by the mean value inequality). Moreover, $(f_k)$ converges  to $f$ with respect to the supremum norm and, hence, $f$ belongs to $\A$.
\end{example}


\subsection{The Royden boundary}\label{Subsection-Boundary}
In this subsection, we consider a {\cpt} graph $(b,c)$ over $X$.  According to the considerations of the previous subsections, there arises  a canonical compactification $K$ of $X$ and  this compactification agrees with the Royden compactification $R$. Here, we are concerned with the arising boundary
\begin{align*}
    \partial X:=K\setminus X = R\setminus X.
\end{align*}

We first note that points in the boundary are exactly the accumulation points of $X$:

\begin{lemma}\label{l:discrete} If $k\in R$,
then $k$ is a discrete point of $R$ if and only if $k\in X$.
\end{lemma}
\begin{proof} We will use the fact that $K=R$ proved above. For the forward direction, note that every singleton set $\{x\}$ is open in $K$ for $x\in X$. For the other direction, we know, by Theorem \ref{t:open}, that $X$ is dense in $K$. This implies that, for every $k\in \partial X$ and every neighborhood $W$ of $k$, there is an $x\in X$ with $x\in W$. This implies that $k$ is not a discrete point.
\end{proof}

We now characterize the elements of $\ow D$ yielding functions which vanish on the boundary.
Recall that $\ow D_o$ is the closure of $C_c (X)$ in $\ow D$ with respect to $\aV{\cdot}_o$ according to Definition \ref{def:Dzero}.

\begin{theorem}\label{l:c_c} If $(b,c)$ is {\cpt} and connected, then
$$\ow D_o = \{ u\in \ow D : \oh u\vert_{\partial X}=0\}. $$
\end{theorem}
\begin{proof}
We first show the inclusion ``$\subseteq$'': We start by  proving that $\oh v\vert_{\partial X}=0$ for all $v\in C_{c}(X)$.
By Theorem~\ref{t:open} and Lemma~\ref{l:discrete} for every $k\in\partial X$ there exists $(x_{n})$ in $X$ with $x_{n}\longrightarrow k$, $n\longrightarrow \infty$.   Since $(x_{n})$ eventually leaves every finite set, we infer that $v(x_{n})=0$ for $n$ sufficiently large for all $v \in C_c(X)$. Hence, $\oh v\vert_{\partial X}=0$ for all $v\in C_{c}(X)$.

Now, let $u\in  \ow D_o$ and let $(v_{n})$ in $C_{c}(X)$ be converging to $u$  with respect to $\aV{\cdot}_{o}$ and, thus, by Lemma~\ref{l:cpt}, with respect to $\aV{\cdot}_{\infty}$. This implies that $\aV{\oh v_{n}-\oh u}_{\infty}\longrightarrow0$, $n\longrightarrow\infty$, and, in particular, $\oh u\vert_{\partial X}=0$.

\smallskip

We now turn to the opposite inclusion ``$\supseteq$'': Let $u\in \ow D$ with $u\vert_{\partial X}=0$ be given. Without loss of generality, we can assume that $u\geq 0$. Now, for $\varepsilon >0$, consider the map
$$C_\varepsilon : \R \longrightarrow \R,$$
with $C_\varepsilon (x) = 0$ for $x\leq \varepsilon$ and $C_\varepsilon (x) = x - \varepsilon$ for $x\geq \varepsilon$. It is not hard to see  $C_\varepsilon u$ belongs to $\ow D$ as well (as $C_\varepsilon (p) \leq p$ and $|C_\varepsilon (p) - C_\varepsilon (q)| \leq |p - q|$ for all $p,q\in \R$).

\smallskip

\textit{Claim: $C_\varepsilon u$ has finite support.}
Assume the contrary. Then, there exists a  sequence $(x_n)$ in $X$ of pairwise different points  with $u (x_n) \geq \varepsilon$. By the compactness of $R$, we can assume, without loss of generality, that $(x_n)$ converges to some $k$.  By Lemma~\ref{l:discrete}, this $k$ must belong to the boundary $\partial X$. On the other hand, by $u (x_n)\geq \varepsilon$ we have $u (k)\geq \varepsilon$. This is a contradiction to the vanishing of $u$ on the boundary.  This proves the claim.

 \smallskip

 \textit{Claim: $C_\varepsilon u \longrightarrow u$ with respect to $\aV{\cdot}_{o}$.} It is not hard to see that $C_\varepsilon u$ converges pointwise to  $u$. Moreover, a short calculation gives
 \begin{eqnarray*}
 | (C_\varepsilon  u  - u) (x) -  (C_\varepsilon u - u ) (y)|^2 &\leq & 2 | C_\varepsilon u (x) - C_\varepsilon u (y)|^2 + 2 | u (x) - u(y)|^2\\
 &\leq & 4 |u(x) - u(y)|^2.
  \end{eqnarray*}
 Combining this with the pointwise convergence, we infer from the Lebesgue dominated convergence theorem that
 $$\ow Q (C_\varepsilon u - u)\longrightarrow 0,\varepsilon \longrightarrow 0.$$
 This proves the claim.

\smallskip

With the preceding two claims the proof of the theorem is finished.
\end{proof}

We finish this section be discussing a slight generalization of $\aV{\cdot}_o$.
 Assume that $(b,c)$ is \cpt. Let $k\in R$. Define
\begin{align*}
    \as{f,g}_{k}=\ow Q(f,g)+\ov{\oh f(k)}\oh g(k),\quad f,g\in \ow D
\end{align*}
and
$\aV{f}_{k}:=\as{f,f}_{k}^{\frac{1}{2}}$. We call $(\ow D,\as{\cdot,\cdot}_{k})$, $k\in K$, the \emph{generalized Yamasaki space}.

\begin{theorem}\label{p:Yamasaki2} If $(b,c)$ is {\cpt} and connected, then $\aV{\cdot}_{k}$ and $\aV{\cdot}_{o}$ are equivalent for all $o,k\in R$. In particular, $(\ow D,\as{\cdot,\cdot}_{k})$ is a Hilbert space for all $k\in R$.
\end{theorem}
\begin{proof}
We will use that $K$ agrees with $R$.
It suffices to prove the statement for  $o\in X$ and  $k\in K$. Let $(k_{n})$ be a net  in $X$  converging to $k$. Observe that
\begin{align*}
|\oh f(k)|&\leq|\oh f(k)- f(o)|+| f(o)|\leq
\lim_{{n}\longrightarrow\infty}|f(k_{n})-f(o)|+|f(o)|\\
&\leq
\ow Q(f)^{\frac{1}{2}}\lim_{n\longrightarrow\infty}\varrho(o,k_{n})  +|f(o)|.
\end{align*}
Since, by Theorem~\ref{characterization-cpt}, $\diam_{\varrho}(X)<\infty$ we conclude the equivalence of the norms. Since, by Proposition~\ref{p:Yamasaki}, $(\ow D,\as{\cdot,\cdot}_{o})$ is a Hilbert space for all $o\in X$, we conclude the `in particular'.
\end{proof}





\subsection{Relationship  between  $R$ and metric completions of $X$}
In this subsection, we study how $R$ is related to the  metric completions studied in the preceding subsections.

\bigskip

\begin{theorem}\label{l:embedding-varrho} Let  $(b,c)$  be {\cpt} and connected and $\ov{X}^{\varrho}$ be the metric completion of $X$ with respect to  $\varrho$. Then,  there exists a unique continuous map  $\kappa :\ov{X}^{\varrho}\longrightarrow R$ extending the identity $X\longrightarrow X, x\mapsto x$. For any $f\in \ow D$, its extension $f^{(\varrho)}$ to $\ov{X}^{\varrho}$ is given by $\widehat{f}\circ \kappa$.
The map $\kappa$  is one-to-one. If $\ov{X}^\varrho$ is compact, it is a homeomorphism.
 \end{theorem}
\begin{proof}  We will show a series of claims.

\smallskip

\textit{Claim: Such a $\kappa$ is unique.}
 By denseness of $X$ in $\ov{X}^{\varrho}$,  uniqueness of such a map $\kappa$ is clear.

\smallskip

\textit{Claim: There exists such a $\kappa$:}
We will make use of the fact that $R$ coincides with $K$ (Theorem~\ref{KR}). We start with a slight reformulation of the Gelfand theory already discussed above:
Let $\widetilde{K}$ be the set of all multiplicative linear maps from $\A $ to $\C$. Then, $\widetilde{K}$ consists of the set of characters, i.e., non-vanishing multiplicative functionals  $K_0$ together with the zero multiplicative functional $0_M$.  If $1\in \A$, then $K_0$ is compact and is $\A$ is canonically isometrically isomorphic to $C(K_0)$.  If $1$ does not belong to $\A$, then $K_0$ is not compact, but $\widetilde{K}$ is  and $\A$ is then  canonically isometrically isomorphic to the algebra of continuous functions on $\widetilde{K}$ vanishing at $0_M$.  In this case, $\widetilde{K}$ is homeomorphic to $K$ via the unique map which is the identity on $K_0$ and maps $\gamma_\infty$ to $0_M$.

To show existence of the desired map $\kappa$, it therefore suffices to provide a continuous map from $\ov{X}^\varrho$ to $\widetilde{K}$  extending the identity such that  $0_M$ is not in its range if $1$ belongs to $\A$.

By definition of $\varrho$, every  function $f\in \ow D$ is Lipschitz continuous with respect to $\varrho$  with constant $\ow Q (f)^{\frac{1}{2}}$. Hence, it can be uniquely extended to  a continuous function on $\ov{X}^{\varrho}$. The same is then true for any function $f\in \A$ (as $\A$ is the closure of $\ow D$ with respect to the supremum norm).
 Denote the extension of such a function $f\in \A$ to $\ov{X}^{\varrho}$ by $\widetilde{f}$. Then, clearly  every $k\in \ov{X}^{\varrho}$ defines a unique  multiplicative linear functional $\kappa (k)$ on $\A$ with $\kappa (k) (f) = \widetilde{f} (k)$.  This shows that $\kappa (k)$ is either $0_M$ or an element of $K_0$ and hence belongs to $\widetilde{K}$. Moreover, if $1$ belongs to $\A$, then $\kappa (k)$ can not agree with $0_M$ as it obviously takes the value $1$ on the constant function $1$.

 Moreover,  it is not hard to see that $\kappa$ is continuous since whenever $(k_n)$ converges to $k$ (w.r.t. $\varrho$), then $\kappa (k_n) (f)$ will converge to $\kappa (k) (f)$ for every $f\in \A$. In fact,  this is just continuity of the functions $\widetilde{f}$.

This gives existence of the desired map $\kappa$.

\smallskip

\textit{Claim: The map $\kappa$ is one-to-one:} Let $p,q\in\ov{X}^\varrho$ be given with $p\neq q$. Then, $\varrho (p,q)>0$ holds.  By the definition of $\varrho$, there then exists $f\in \ow D$ with $\ow Q (f) \leq 1$ and
$$|f(p) - f(q)|\geq \frac{1}{2} \varrho (p,q) >0.$$
This immediately implies that $\kappa (p) $ and $\kappa (q)$ are not equal (as they take different values  on $f$).

\smallskip

\textit{Claim: The map $\kappa$ is a homeomorphism if $\ov{X}^\varrho$ is compact:}
 We have already seen continuity and injectivity. It remains to show that $\kappa$ is onto and that the inverse is continuous. By compactness of $\ov{X}^\varrho$, it suffices to show that $\kappa$ has dense range. This, however, is clear as $X$ lies in the range of $\kappa$.


\smallskip

It remains to show the statement on the functions $f$. By the continuity of $\kappa$ and the definition of $\widehat{f}$, the function $\widehat{f}\circ \kappa$ is continuous on $\ov{X}^{\varrho}$. By a short calculation, it can be seen to agree with $f$ on $X$. Hence, it must be equal to the (unique) extension of $X$ to $\ov{X}^{\varrho}$. This proves the remaining  statement of the lemma.
\end{proof}

\textbf{Remark.} In general, $(b,c)$ being canonically compactifiable does not imply that $\overline{X}^{\varrho}$ is compact, as can be seen from Theorem~\ref{implication-D} together with  Example~\ref{e:C_does_not_imply_B}.
More precisely, the mentioned theorem shows that $(C)$ implies  canonical compactifiability,  that is $(D)$, whereas the example shows that $(C)$ does not imply the compactness of $\overline{X}^{\varrho}$, which is $(B)$.
\medskip

As a consequence  of the previous theorem, we note the following.

\begin{corollary} Let  $(b,c)$  be {\cpt} and connected and $\ov{X}^d$ be the metric completion of $X$ with respect to $d$. Then, there exists a unique continuous map $\ov{X}^d\longrightarrow R $ extending the identity $X\longrightarrow X, \,  x\mapsto x$. It is given by $\kappa \circ \iota$.
\end{corollary}
\begin{proof}
 Uniqueness is clear from denseness of $X$ in $\ov{X}^d$. Existence follows from the previous lemma  and Theorem \ref{extension-d-to-varrho}.
\end{proof}

In the case of intrinsic metrics, we can say more.

\begin{theorem} \label{embedding-intrinsic} Let  $(b,c)$  be {\cpt} and connected with $C:=\sum_{x\in X} c(x) < \infty$.  Let $m$ be a measure on $X$ with $m(X) < \infty$ and $\sigma$ be an intrinsic metric with respect to $m$ such that $\ov{X}^\sigma$ is compact.
Then, there exists a unique continuous map
$$\lambda =\lambda_\sigma : R\longrightarrow \ov{X}^\sigma$$
extending the identity on $X$  and this map  is onto.
\end{theorem}

\begin{proof}
Again, using Theorem~\ref{KR}, we can work with $K$ instead of $R$. Uniqueness of $\lambda$ is clear.  By the usual denseness argument surjetivity of $\lambda$ is clear from compactness of $R$ and continuity of $\lambda$. Thus, it remains to show existence.
Consider the set $\Lip (\ov{X}^\sigma)$ of Lipschitz  functions with respect to $\sigma$. By compactness of $\ov{X}^\sigma$, any such function is bounded.  Moreover, whenever   $f$  belongs to $\Lip (\ov{X}^\sigma)$ there exists a constant $L$ with
$$|f(x) - f(y)| \leq L\sigma (x,y)$$
 and hence $f$
satisfies
\begin{eqnarray*}
& & \frac{1}{2} \sum_{x, y\in X } b(x,y) |f(x) - f(y)|^2 + \sum_{x\in X} c(x) |f(x)|^2 \\
&\leq & \frac{1}{2} \sum_{x, y \in X} b(x,y) L^2 \sigma^2 (x,y) + C \|f\|_\infty\\
&  \leq  & L^2 m(X) + C \|f\|_\infty <  \infty.
\end{eqnarray*}
This shows that $\Lip (\ov{X}^\sigma)$ can be seen as a subset of $\ow D$ and hence a subalgebra of $\A$.
Thus, we have a natural isometric homomorphism
$$ j: \Lip (\ov{X}^\sigma)\longrightarrow \A, \;\: f\mapsto f|_X,$$
where $f|_X$ denotes the restriction to $X$.
Now, the constant function obviously belongs to $\Lip (\ov{X}^\sigma)$. Moreover, $\Lip (\ov{X}^\sigma)$ clearly separates points of $\ov{X}^\sigma$ as it contains  $\sigma(x,\cdot)$ for any $x\in X$. Hence, by the Stone-Weierstrass Theorem, the algebra $\Lip (\ov{X}^\sigma)$ is dense in $C(\ov{X}^\sigma)$. As $j$ is isometric, it can therefore be extended to an isometric map (again denoted by $j$)
$$j: C(\ov{X}^\sigma) \longrightarrow \A = C(K).$$
By standard results from Gelfand theory,
this map $j$ now induces a  continuous map $j^\ast :K \longrightarrow \ov{X}^\sigma$ with $f(j^\ast (k) )  = k(j(f))$. Then, $\lambda = j^\ast$ is the desired map.
\end{proof}



\subsection{A key example: If $1/b$ and $c$ are  summable} \label{key-example}
The results of the previous subsections have given relationships between the various metric completions of a graph $(b,c)$ over $X$. As the counterexamples in the last section show, these completions will be different in general. In this subsection, we study a particular class of examples in which essentially `all' completions agree.
More precisely, we consider connected  graphs $(b,c)$ over $X$ with
$$ B:=\frac{1}{2}\sum_{x,y\in X: b(x,y)\neq 0} \frac{1}{b(x,y)} < \infty.$$
In this case, we can make a thorough study of compactness properties of $\overline{X}^d$ and $\overline{X}^\varrho$.

Whenever $(b,c)$ is a connected graph over $X$ satisfying $B<\infty$, there is a canonical finite measure $M$ on $X$ with total mass $B$  defined by
$$M : X\longrightarrow (0,\infty), \;M(x) = \frac{1}{2} \sum_{y\in X : b(x,y) \neq 0} \frac{1}{b (x,y)}.$$
Note that $M$ vanishes nowhere due to connectedness.

\smallskip

 Now, there are some crucial consequences to $B<\infty$: One consequence is  that $d$ is an intrinsic metric  with respect to a finite measure, namely, $M$. Another most important consequence is that $X$ is totally bounded with respect to $d$.  This is the content of the next proposition.

\begin{proposition} \label{compactness-xd}  Let $(b,c)$ be a connected graph over $X$.  Assume $B<\infty$ and let $M$ be the associated measure. Then, the following hold.

\begin{itemize}

\item[(a)]   The space  $\overline{X}^d$ is compact and so is $\ov{X}^\varrho$.

\item[(b)] The metric  $d$ is an intrinsic metric with respect to $M$.

\end{itemize}

\end{proposition}
\begin{proof}  (a)   To show the compactness of $\overline{X}^d$, it suffices to show that, for any $\varepsilon >0$, the space $X$ can be covered by finitely many $\varepsilon$-balls (w.r.t. d). This follows easily from $B<\infty$.

By Theorem \ref{extension-d-to-varrho}, the map $\iota$ is continuous and onto and $\overline{X}^\varrho$ is then compact as well.

\smallskip

(b) Note that  $d(x,y) \leq \frac{1}{b(x,y)}$ whenever $x$ and $y$ are neighbors.  Then, a direct calculation gives
$$ \frac{1}{2} \sum_{y\in X} b(x,y) d(x,y)^2 \leq \frac{1}{2} \sum_{y\in X : b(x,y) \neq 0}\frac{1}{b (x,y)} = M (x).$$
This finishes the proof.
\end{proof}

\textbf{Remarks.} (a)  Let us note that if $B<\infty$ and $c\equiv0$ hold, then   the metric $d$ is -- in a certain sense -- the maximal  metric which is intrinsic with respect to a finite measure. More precisely, whenever  $\sigma$ is an intrinsic metric with respect to a measure $m$ with total mass $m(X)$, then    $\sigma \leq \sqrt{m(X)}  \varrho$ by Theorem \ref{characterization-varrho-via-sigma}. Moreover, by Lemma \ref{l:DLip}, we have $\varrho^2 \leq d$. Putting this together, we infer $\sigma \leq \sqrt{m(X)} d^{1/2}$.

(b) A further discussion of $\overline{X}^d$  in the situation $B<\infty$ is given below in Corollary~\ref{corFC}.

\medskip

From this proposition we immediately infer the first main result of this section.

\begin{corollary} Let $(b,c)$ be a connected graph over $X$. Assume $B<\infty$. Then, the unique canonical map $\kappa : \ov{X}^\varrho \longrightarrow R$ extending the identity on $X$ is an homeomorphism.
\end{corollary}
\begin{proof}  By (a) of the previous proposition, $\ov{X}^\varrho$ is compact. Thus,  its diameter is finite. Theorem \ref{characterization-cpt} then    gives that $(b,c)$ is {\cpt} and the desired statement now follows from Theorem \ref{l:embedding-varrho}.
\end{proof}

Under the additional assumption $$C:=\sum_{x\in X} c(x) <\infty$$ we can even infer equality of all three compact spaces involved.

\begin{theorem} Let $(b,c)$ be a connected  graph over $X$ with $\sum_{x\in X} c(x) < \infty$ and  $\sum_{x,y\in X: b(x,y)\neq 0} \frac{1}{b(x,y)} < \infty$. Then, the unique  continuous map  $\iota : \overline{X}^d\longrightarrow \overline{X}^\varrho$ extending the identity on X is a homeomorphism. In particular, all three spaces $\ov{X}^d$, $\ov{X}^\varrho$ and $R$ are compact and homeomorphic via the  unique continuous maps extending the identity on $X$.
\end{theorem}
\begin{proof} By Proposition \ref{compactness-xd},  the metric $d$ is an intrinsic metric with compact $\ov{X}^d$. Hence,  Theorem \ref{embedding-intrinsic} gives a unique continuous map $\lambda : R\longrightarrow \ov{X}^d$ extending the identity on $X$ and this map is onto. By Theorem \ref{l:embedding-varrho}, we furthermore  have a unique continuous map $\kappa : \ov{X}^\varrho \longrightarrow R$  extending the identity on $X$  and this map is onto as well. Finally, by Theorem \ref{extension-d-to-varrho}, there is a unique map $\iota : \ov{X}^d\longrightarrow \ov{X}^\varrho$ extending the identity on $X$ and this map is onto as well. Thus, we end up with the following diagram of maps
$$ \ov{X}^d \stackrel{\iota}{\longrightarrow} \ov{X}^\varrho \stackrel{\kappa}{\longrightarrow} R \stackrel{\lambda}{\longrightarrow} \ov{X}^d$$
each map being onto and the unique continuous extension of the identity on $X$. This gives that their composition $ \lambda \circ \kappa \circ \iota$
must be the identity and all maps must be one-to-one. Thus, they are all bijective and hence homeomorphisms (as the underlying spaces are compact).
\end{proof}

It was proved in \cite{Geo} that if $B<\infty$ is satisfied, then $\overline{X}^d$ is canonically homeomorphic to a well-known space called the \emph{end-compactification} or \emph{Freudenthal compactification} of $X$. Combined with the last theorem and Theorem~\ref{KR}, this yields the following corollary.

\begin{corollary} \label{corFC}
Let $(b,c)$ be a connected  graph over $X$ with $\sum_{x\in X} c(x) < \infty$ and  $\sum_{x,y\in X: b(x,y)\neq 0} \frac{1}{b(x,y)} < \infty$. Then each of the spaces $\ov{X}^d$, $\ov{X}^\varrho, K$ and $R$ is homeomorphic to the end-compactification of $X$. 
\end{corollary}

In particular, we obtain that the boundaries of $\ov{X}^d$, $\ov{X}^\varrho, K$ and $R$ are totally disconnected in this case. \\

\textbf{Remark.} Assume  the  situation of the previous theorem. Then, the $d$-diameter of $X$ is finite. Thus, any  metric $\sigma$ which is intrinsic with respect to a finite measure is bounded by Corollary \ref{c:bound-intrinsic}. By
  Theorem~\ref{rho-to-sigma}, there then exists, for any such metric $\sigma$, a   unique continuous map $\gamma: \overline{X}^{\varrho}\longrightarrow \overline{X}^{\sigma}$ and this map is onto.

\subsection{A little summary}
We summarize a large part of the preceding considerations in the following result.

\begin{theorem}\label{implication-D} Let $(b,c)$ be a graph over $X$ and consider the following statements:

\begin{itemize}
\item[$(A)$] $X$ is totally bounded with respect to $d$.
\item[$(B)$]  $X$ is totally bounded with respect to $\varrho$.
\item[$(C)$]  $X$ is totally bounded with respect to any  metric $\sigma $ which is intrinsic with respect to a finite measure.
\item[$(D)$]  $\ow D$ consists only of bounded functions.
\end{itemize}
Then, the implications $(A)\Longrightarrow (B)\Longrightarrow (D)$ hold. If $c\equiv0$, furthermore the implications $(B)\Longrightarrow (C) \Longrightarrow (D)$ hold.

Moreover,  the following statements hold:
\begin{itemize}
\item $(A)$ implies the existence and uniqueness of a continuous surjective map $\iota : \ov{X}^d\longrightarrow \ov{X}^\varrho$  extending the identity on $X$.
    \item $(B)$ implies the existence and uniqueness of  a homeomorphism $\kappa : \ov{X}^\varrho \longrightarrow R$ extending the identity on $X$.

        \item If $c\equiv0$, then $(B)$ implies existence and uniqueness of a continuous surjective map $\gamma : \ov{X}^\varrho \longrightarrow \ov{X}^\sigma$ for any metric $\sigma$ which is intrinsic with respect to a finite measure.

\item If $c\equiv0$, then $(C)$ implies existence and uniqueness of a continuous surjective map $\lambda : R\longrightarrow \ov{X}^\sigma$ for any metric $\sigma$ which is intrinsic with respect to a finite measure.

\end{itemize}

\end{theorem}
\begin{proof} We first deal with the two chains of implications.

The implication $(A)\Longrightarrow (B)$ follows from Lemma \ref{l:DLip}. The implication $(B)\Longrightarrow (D) $ follows from Theorem~\ref{characterization-cpt}.

  Under the assumption $c\equiv0$, the  implication $(B)\Longrightarrow (C)$ follows  from Theorem~ \ref{rho-to-sigma}.  Under the assumption $c\equiv0$, the implication $(C)\Longrightarrow (D)$ follows from Corollary \ref{sufficient-condition-cpt-intrinsic}.

\smallskip

We now turn to the proofs of the four last statements:

\smallskip

The first statement follows from Theorem~\ref{extension-d-to-varrho}. The second statement follows from Theorem~\ref{l:embedding-varrho}. The third statement follows from Theorem~\ref{rho-to-sigma}.
The  forth statement  follows  from Theorem~\ref{embedding-intrinsic}.
\end{proof}


\section{Operator theory: Discrete spectrum and Dirichlet problem}\label{Section-Operator}
In this section, we use the canonical compactification discussed above to study operator theory.
 We will show that the corresponding operators have discrete spectrum and that the Dirichlet Problem has a unique solution.  Moreover, we  will explicitly describe the domain of the Dirichlet Laplacian.


\subsection{Discrete spectrum}\label{Discrete}
In this subsection, we give criteria for an operator on a graph of finite measure to have compact resolvent.  In this case, the spectrum of the operator is purely discrete.

\medskip

\begin{theorem}\label{t:ultracontractive} Let the graph  $(b,c)$  over $X$ be  {\cpt} and connected.  Let  $L$ be the  operator associated to a closed, symmetric form $Q$ with $Q^{(D)}\subseteq Q\subseteq Q^{(N)}$. Then,
the heat semigroup and the resolvent are ultracontractive, i.e., $e^{-tL}$, $t>0$, and $(L+\al)^{-1}$, $\al>0$, are bounded operators from  $\ell^{2}(X,m)$ to $\ell^{\infty}(X)$.
If, additionally, $m(X)<\infty$, then
$e^{-tL}$, $t>0$, and $(L+\al)^{-1}$, $\al>0$,  are trace class.
\end{theorem}
\begin{proof} We only consider the semigroup operators $e^{-t L}$, $t >0$. The statements on resolvents can then be derived by standard techniques.

\smallskip

Let $t>0$ be arbitrary. We have
$$e^{-t L} \ell^{2}(X,m)\subseteq D(L) \subseteq D(Q)\subseteq \ow D\subseteq \ell^{\infty}(X).$$
Since $e^{-tL}$ is a continuous operator  on $\ell^2 (X,m)$ we now  obtain, by a simple application of the closed graph theorem, that $e^{-t L} $  can be seen as a continuous map from $\ell^2 (X,m)$ to $\ell^\infty (X)$. This shows the first statement.  If $m(X) <\infty$, there is a canonical continuous  embedding
$$j : \ell^\infty (X)\longrightarrow \ell^2 (X,m), \, f\mapsto f.$$
Then, $e^{-t L } = j e^{-t L}$  is a composition of a continuous maps from $\ell^2 (X,m) $ to $\ell^\infty (X)$ with a continuous map from $\ell^\infty (X)$ to $\ell^2 (X,m)$. Thus, it is a Hilbert-Schmidt operator by factorization principle (or a direct calculation).  Then,
$$e^{-t L }  = e^{- \frac{t}{2}  L} e^{- \frac{t}{2}L}$$
is trace class as a product of two Hilbert-Schmidt operators.
\end{proof}

\begin{corollary}\label{t:discrete_spectrum} Let $(b,c)$ be \cpt, connected and $m(X)<\infty$. If $L$ is the  operator associated to a closed, symmetric form $Q$ with $Q^{(D)}\subseteq Q\subseteq Q^{(N)}$, then
the spectrum of $L$ is purely discrete.
\end{corollary}

\begin{proof} This  follows directly  from Theorem~\ref{t:ultracontractive} and the spectral mapping theorem.
\end{proof}

\textbf{Remarks.}  (a) If  the form $Q$ in the previous corollary  is additionally assumed to be a Dirichlet form, then Theorem~2.1.4 and Theorem~2.1.5 of \cite{Davies} give that
\begin{itemize}
\item  $e^{-tL}$, $t>0$, and $(L+\al)^{-1}$, $\al>0$,  are norm analytic and compact on all $\ell^{p}(X,m)$,  $1\leq p\leq\infty $,
\item  the spectra of the generators of $e^{-tL}$ on $\ell^{p}(X,m)$ agree for all  $1\leq p\leq\infty$.
\end{itemize}

(b) There has been quite some recent interest in graphs whose Laplacians have purely discrete spectrum, see, e.g.,  \cite{Gol,Kel,KL3, KLW}.

\bigskip

A general discussion of various equivalent characterizations of purely discrete spectrum in the context of perturbations by potentials  for arbitrary selfadjoint operators in a measure space can be found in
 \cite{LStW}. Here,  we point out the following alternative formulation (and proof) of the previous corollary.

\begin{theorem}\label{t:compact_resolvent} Let $(b,c)$ be a {\cpt} graph over $X$ equipped with a measure $m$ with $m(X)<\infty$. Let  $Q$ be a closed form  with domain $D$ and  $Q^{(D)}\subseteq Q\subseteq Q^{(N)}$. Then, the embedding
$$(D, \aV{\cdot}_Q) \longrightarrow \ell^2 (X,m)$$
is compact.
\end{theorem}
\begin{proof}
Clearly, $\aV{\cdot}_{Q}$ is a restriction of the norm $\aV{\cdot}_{\ow Q}$ on $\ow D$.  Let $(f_n)$ be a  sequence in $D$ which is  uniformly bounded with respect to  $\aV{\cdot}_Q$ and thus, by Lemma \ref{l:equivalent_norms}, with respect to $\aV{\cdot}_o$.  As $(b,c) $ is {\cpt},   Lemma \ref{l:cpt} then shows that  the sequence $(f_n)$ is uniformly bounded with respect to the supremum norm. Thus, it contains a pointwise converging subsequence. By $m(X)<\infty$, this sequence must then converge in $\ell^2 (X,m)$. This proves the theorem.
\end{proof}

We end this subsection with a general lemma, which will also be used in the characterization of the Dirichlet Laplacian (Theorem \ref{t:Dirichlet_bc}).  Recall that $\aV{f}_{\ow Q} = (\ow{Q}(f) + \aV{f}^2)^{\frac{1}{2}}$ and  $\aV{f}_{o} = (\ow{Q}(f) + |f(o)|^2)^{\frac{1}{2}}$ for $f \in \ow D$ and fixed $o \in X$.

\begin{lemma}\label{l:equivalent_norms}
Let $(b,c)$ be a {\cpt} graph over $X$ equipped with a measure $m$ with $m(X)<\infty$.  Then, $\aV{\cdot}_{\ow Q}$ and $\aV{\cdot}_o$ are equivalent norms on $\ow D$.
\end{lemma}
\begin{proof}
Obviously,
$$\aV{f}_o \leq \left(1 + m(o)^{-\frac{1}{2}} \right) \aV{f}_{\ow Q}.$$
On the other hand, by Lemma \ref{l:cpt} and $m(X)<\infty$, there exists a $C>0$ with
$$ \|f\|_\infty \leq C \aV{f}_o.$$
Hence,
$$\aV{f}_{\ow Q}^2 = \ow Q (f) + \aV{f} \leq  \ow Q (f) + m(X) C^2 \|f\|_o^2  \leq (1+ m(X)C^2) \aV{f}_o^2.$$
\end{proof}


\subsection{The Dirichlet problem}\label{Boundary}
In this subsection, we consider the boundary value problem
\begin{align*}
\LL u=0 \:\; \mbox{on $X$}\:\; \quad \oh u\vert_{\partial X}=\ph
\end{align*}
on a \cpt\  graph $(b,c)$ over $X$. Note that there is no measure involved in our discussion. For somewhat similar considerations in the context of metric graphs we refer the reader to \cite{Car}. Complementary results on when there are no (bounded) harmonic functions can be found \cite{HuK}.

 \bigskip

Whenever we have a \cpt\ graph $(b,c)$ over $X$,  we denote the set of continuous functions on its boundary by $C (\partial X)$. Moreover, we set
$C_0(\partial X ):= C(\partial X)$ if $ \sum_{x\in X}  c(x) <\infty$ and
$C_0(\partial X )=\{ \ph \in C (\partial X) : \gamma_\infty (\ph) =0\}$, otherwise.
Then, the main result of this subsection reads as follows.

\begin{theorem}\label{t:bvp}(Uniqueness and existence of solutions to the DP) Let $(b,c)$ be a \cpt\ graph over $X$.
Then, the equation
\begin{align*}
\LL u=0 \:\; \mbox{on $X$}\:\; \quad \oh u\vert_{\partial X}=\ph.
\end{align*}
has a unique solution $u\in\A$ for all  $\ph\in C_0 (\partial X)$.
\end{theorem}

\textbf{Remark.}
Carlson \cite[Theorem~4.5]{CarAft} proved that the Dirichlet problem for continuous functions on the boundary of $\overline{X}^d$ is solvable when $(A)$ and a further strong condition hold. In comparison, Theorem~\ref{t:bvp} assumes only the much weaker $(D)$ but the boundary on which the Dirichlet problem is solved is, in general, smaller. Note that, assuming $(A)$, solvability of the Dirichlet problem on the boundary of $\overline{X}^d$ implies solvability on $\partial X$, since by the first two statements of the second part of Theorem~\ref{implication-D} there is a a continuous surjection from $\overline{X}^d$ to $R$ extending the identity on $X$. It would be interesting to know whether the Dirichlet problem  on the boundary of $\overline{X}^d$ is solvable for every graph satisfying $(D)$; this issue is further  discussed in the last section.

\medskip

The proof of Theorem~\ref{t:bvp} will be given at the end of this subsection after some preliminary considerations.  In particular, for the proof of the theorem  we need a maximum principle which states that every harmonic function takes its maxima and minima on the boundary. As this may be of independent interest, we present  a slightly more general discussion.

\smallskip

We start with two statements on \textit{subharmonic functions}, i.e., functions satisfying $\LL f \leq 0$.

\begin{proposition} Assume that $(b,c)$ is connected. If $f\in \widetilde{F}$  is non-negative and not constant and satisfies  $\LL  f\leq 0$ on $X$, then $f$ does not attain a maximum on $X$.
\end{proposition}
\begin{proof}  Let $f$ be as in the statement. Assume that $f$  attains its maximum  on $X$ at $k\in X$. By  $\LL  f \leq 0$,  we then have
$$ 0 \geq \sum_{x\in X} b(k,x) (f(k) - f(x)) + c(k) f(k).$$
By maximality of $f(k)$ and non-negativity of $f$ we infer $f(x) = f(k)$ for all neighbors $x$ of $k$. Continuing this argument and using connectedness,   we infer that $f$ is constant (and equal to $f(k)$) and this is a contradiction.
\end{proof}

\begin{lemma} Let $(b,c)$ be a graph over $(X,m)$. If $f\in \widetilde{F}$ satisfies $\LL  f = 0$, then $|f|$ satisfies $\LL  |f| \leq 0$.
\end{lemma}
\begin{proof} This follows by a direct computation: $ \LL f =0$ implies $\alpha \LL f = 0$ for all complex numbers $\alpha$. Thus, for any $k\in X$ and complex number $\alpha$, we obtain
$$ 0 = \sum_{x\in X} b(k,x) (\alpha f(k) - \alpha f(x)) + c(k) \alpha f(k).$$
Choosing $\alpha$ with $|\alpha| = 1 $ and $\alpha f (k) =|f(k)|$, we infer
$$0 = \sum_{x\in X} b (k,x) (|f(k)| - \alpha f(x)) + c(k) |f(k)|$$
and the desired statement follows easily.
\end{proof}

\begin{corollary}\label{p:maximumprinciple}(Maximum  principle) Assume that $(b,c)$ is \cpt\ and connected. If $f\in \A$  satisfies $\LL f=0$ on $X$, then
\begin{align*}
\aV{f}_{\infty}=    {\| \oh f\vert_{\partial X}\|}_{\infty}.
\end{align*}
\end{corollary}
\begin{proof} As $f$ belongs to $\A$, the function $|f|$  attains its maximum on $K$. Combining the previous lemma and the preceding proposition, we infer that $|f|$ attains its maximum on $\partial X$. This directly gives the statement of the corollary.
\end{proof}

\begin{proof}[Proof of Theorem~\ref{t:bvp}] We only consider the case where $\sum_{x\in X}  c(x) < \infty$.  The other case can be treated similarly.  We have to show existence and uniqueness of solutions.

\smallskip

\textit{Uniqueness of solutions.}  Uniqueness is a direct consequence of Corollary \ref{p:maximumprinciple}.

 \smallskip

 \textit{Existence of solutions.}
 Consider the set $G\subseteq  C(\partial X)$ of functions $\ph$ such that there exists $f\in \ow D$ with $\ph=\oh f\vert_{\partial X}$.

\emph{Step~1.} For all $\ph \in G$ there is a unique solution of the boundary value problem.

\emph{Proof of Step~1.}  Fix $k \in \partial X$ with $k\neq \gamma_\infty$.  Recall that  $\langle \cdot,\cdot\rangle_k$ is an inner product on $\ow D$ making it into a Hilbert space by Theorem~\ref{p:Yamasaki2}. Let $\ph\in G$. Consider
$$C:=\{w\in \ow D : \oh w\vert_{\partial X}=\ph\}. $$
Then, $C$ is clearly convex. Moreover, it is closed with respect to $\aV{\cdot}_{k}$ by Theorem~\ref{p:Yamasaki} and Lemma \ref{l:cpt}.

  By standard  Hilbert space theory, see, e.g., \cite{Weidmann}, there exists  then a minimizer $u$ of $\aV{\cdot}_{k}$ on $C$. Hence, since $\oh v\vert_{\partial X}\equiv0$ for all $v\in C_{c}(X)$ by Theorem~\ref{l:c_c},  we get
\begin{align*}
\aV{u}_{k}^{2} & \leq \aV{u+tv}_{k}^{2} =\aV{u}_{k}^{2} +2t\as{u,v}_{k}+  t^2 \aV{v}_{k}^{2}\\
 &=\aV{u}_{k}^{2} +2t Re \  \ow Q({u,v})+ t^2 \ow Q(v)
\end{align*}
for all $v\in C_{c}(X)$ and all $t\in \R$.
Therefore, the real part of $\ow Q(u,v)$ is 0 for all $v\in C_{c}(X)$ and repeating the argument with $it$ in place of $t$, gives that $\ow Q(u,v) = 0$ for all $v\in C_{c}(X)$. By the discussion on integration by parts in Section~\ref{Quadratic} we have  $\ow D\subseteq \ow F$ as well as  $\ow Q(u,v)=\sum_{x\in X}\ov{\LL u(x)}v(x)$ for all $v\in C_{c}(X)$. This implies that $\LL u=0$.

\emph{Step~2.}  Suppose that $(f_{n}) $ in $\ow D$ solves $\LL f_{n}=0$, $\oh f_{n}\vert_{\partial X}=\ph_{n}$, for some $(\ph_{n})$ in $C(\partial X)$. Suppose, furthermore, that $\ph_{n}\longrightarrow\ph$ uniformly as $n\longrightarrow\infty$. Then, $(f_{n})$ converges uniformly to some $f\in\A$ that solves $\LL f=0$, $\oh f\vert_{\partial X}=\ph$.

\emph{Proof of Step~2.}
By the maximum principle, Corollary~\ref{p:maximumprinciple},
\begin{align*}
\aV{f_{n}-f_{m}}_{\infty}=\aV{\ph_{n}-\ph_{m}}_{\infty}\to0
\end{align*}
which implies the uniform convergence of $f_{n}$ to some $f\in \ell^{\infty}(X)$.
By the definition of $\A$, we have  $f\in \A$.
 Since  $f_{n}\longrightarrow f$ in $\ell^{\infty}(X)$ we conclude that $\LL f_{n}\longrightarrow \ow Lf$ by Lebesgue's theorem of dominated convergence. This finishes the proof of Step 2.

\smallskip

We can now conclude the existence proof as follows:  By construction, the algebra $\ow D$ is dense in $\A$. In particular,  $\ow D$ separates points of $K$. Hence, $G$ separates points of $\partial X$.  Moreover, by  Proposition~\ref{p:one}, the algebra $G$ does not vanish identically on any point of $\partial X$. Now, the Stone-Weierstrass theorem gives that $G$ is dense in $C (\partial X)$.  The statement of the theorem then follows by combining Step~1 and Step~2.
\end{proof}


\subsection{The Laplacian with Dirichlet boundary conditions}\label{s:Dirichlet_bc}
In this subsection, we give an explicit description of the domain of the form $Q^{(D)}$ and of the domain of the associated Laplacian in the case of \cpt\ graphs over discrete finite measure spaces. We will show that functions in the domain are exactly those vanishing on the boundary and use this to give an explicit description of the Laplacian as well.

\bigskip

Let $(b,c)$ be a graph over  $(X,m)$. Recall that
the form $Q^{(D)}$ has the domain
\begin{align*}
D(Q^{(D)})=\overline{C_{c}(X)}^{\aV{\cdot}_{Q}},
\end{align*}
where
$$\aV{f}_{Q}^2 = \ow Q (f) + \aV{f}^2.$$

\begin{theorem}\label{t:Dirichlet_bc}Let $(b,c)$ be a {\cpt}  graph over  $(X,m)$ with $m(X) <\infty$.  Let $Q^{(D)}$ be the associated form with Dirichlet boundary conditions and $L^{(D)}$ be the corresponding operator.  Then, the form domain is given by
$$D(Q^{(D)}) = \{u \in \ow  D : \oh u\vert_{\partial X} = 0\} = \ow D_o.$$
Moreover, the domain of  $L^{(D)}$ is given by
$$ D(L^{(D)}) = \{ f\in \ow D :  \oh f \vert_{\partial X} = 0 \;\:\mbox{and}\;\:  \ow Lf \in \ell^2 (X,m)\}.$$
\end{theorem}
\begin{proof}  We first show the statement on the form domain.  Clearly, $\aV{\cdot}_{Q}$ is a restriction of the norm
$\aV{\cdot}_{\ow Q}$ on $\ow D$.
By Theorem~\ref{l:c_c} and the definition of $\ow  D_o$, it then  suffices to show that the norms
$\aV{\cdot}_o$ and  $\aV{\cdot}_{\ow Q}$ are equivalent on $\ow D$.  But this is precisely the statement of Lemma \ref{l:equivalent_norms}.

 As for the last statement of the theorem, we note that $L^{(D)}$ is a restriction of $\ow L$ and that $\ow D$ is contained in $\ow F$ by the discussion of Subection~\ref{Quadratic}. Now, the statement follows from
  from the definition of $L^{(D)}$ and the formulas for `integration by parts' discussed in Subection~\ref{Quadratic}.
\end{proof}


\section{Geometry: The eigenvalue $0$ and heat kernel convergence}\label{Section-Geometry}
In this section, we first present an analogue of a result of Yau \cite{Yau} on manifolds characterizing the situation of finite measure and then present an application of an analogue to a theorem of Li.

We start with a spectral characterization of the finiteness of $m(X)$ when $c \equiv 0$.

\begin{theorem}\label{t:finite_meassure}(Characterization of finite measure) If $(X,m)$ is a discrete measure space, then the following assertions are equivalent:
\begin{itemize}

\item[(i)] $m(X)<\infty$.

\item[(ii)] The constant functions are contained in $D(Q^{(N)})$ for some (any) graph $(b,0)$ over $(X,m)$.

\item[(iii)] The constant functions are contained in $D (L^{(N)})$ for some (any) graph $(b,0)$ over $(X,m)$.

\item[(iv)] $E_0 = 0$ is an eigenvalue of $L^{(N)}$ for some (any) graph $(b,0)$ over $(X,m)$.

\end{itemize}
Let, in this case,  $(b,0)$ be a graph over $(X,m)$ and let $L$ be the self adjoint operator associated to a closed form $Q$ satisfying $Q^{(D)}\subseteq Q \subseteq Q^{(N)}$. If the constant functions belong to the domain of $Q$, then they are eigenfunctions of $L$ to the eigenvalue $E_0 =0$.
\end{theorem}
\begin{proof}  Let  $(b,0)$ be a graph over $(X,m)$ and let  $Q$ satisfy $Q^{(D)}\subseteq Q\subseteq Q^{(N)}$. Let $L$ be the operator associated to $Q$.  Then, the following three facts are easy to establish:

\textit{Fact $1$:} $E_0 =0$ is an eigenvalue to $L$ if and only if there exists an non-trivial $u\in D(Q)$ with $Q(u) = 0$.

\textit{Proof.}  This is true for any non-negative form $Q$ by the spectral theorem.

\textit{Fact $2$:} $Q(u) =0$ if and only if $u$ is constant on any connected component of the graph.

\textit{Proof.}  This is a direct consequence of $Q$ satisfying $Q^{(D)}\subseteq Q\subseteq Q^{(N)}$, i.e., $Q(u)=\ow Q(u)=\frac{1}{2}\sum_{x,y} b(x,y)|u(x)-u(y)|^{2}$.

\textit{Fact $3$:}  The constant functions belong to $\ell^2 (X,m)$ if and only if $m(X)<\infty$.

\textit{Proof.} This is clear.

Given these three facts, the theorem follows easily.
\end{proof}

\begin{corollary} Let  $(b,c)$ be a graph over the discrete measure space $(X,m)$ and $L^{(N)}$ be the associated Neumann operator.   The multiplicity of the eigenvalue $E_0 =0$ of $L^{(N)}$ is the number of connected components  $W$  of $X$ with $m(W)<\infty$ and $c\equiv 0$ on $W$.
\end{corollary}
\begin{proof}  It suffices to consider the case when $(b,c)$ is connected.
It then suffices to show that $E_0 =0$ cannot be an eigenvalue if $c\not \equiv 0$.  However, this is clear.
\end{proof}

Let $Q$ satisfy $Q^{(D)}\subseteq Q\subseteq Q^{(N)}$ and let $L$ be the corresponding self adjoint operator.
Then, the heat semigroup $e^{-tL}$ defined by the spectral theorem has a kernel $p:X\times X\times [0,\infty)\longrightarrow\C$, called the \emph{heat kernel}, satisfying
\begin{align*}
    e^{-tL} f(x)=\sum_{y\in X}p_{t}(x,y)f(y)m(y)
\end{align*}
for all $f\in  \ell^2(X,m)$. By \cite[Theorem~7.3]{HKLW} the heat kernel is positive on every connected component.

The following is a consequence of an analogue of a theorem of Li \cite{Li} for graphs (see \cite{HKLW,KLVW} which give an analogue to a result of \cite{CK} from which the theorem of Li follows).

\begin{corollary} Let $(b,0)$ be a connected  graph and $m(X) <\infty$. Let $Q$ be a form satisfying $Q^{(D)}\subseteq Q \subseteq Q^{(N)}$ and let $L$ be the associated operator. If the constant functions belong to the domain of $Q$, then
$$p_t (x,y) \longrightarrow  \frac{1}{m(X)},\;\: t\longrightarrow \infty,$$
for all $x,y\in X$. In particular, this holds for $Q=Q^{(N)}$.
\end{corollary}
 \begin{proof} By the assumptions and Theorem~\ref{t:finite_meassure} of this section, we obtain that  $\Phi = \frac{1}{\sqrt{m(X)}} 1$ is a normalized positive eigenfunction of $L$.
By an analogue of a theorem of Chavel and Karp for graphs (see \cite[Theorem~8.1]{HKLW},\cite[Corollary~5.6]{KLVW}) there exists a unique non-negative $\Phi$ on $X$ such that
$$e^{t E_0} p_t (x,y) \longrightarrow \Phi (x) \Phi (y),\quad t\longrightarrow \infty$$
for all $x,y\in X$. Here, $\Phi \equiv 0$ if the $E_0$ is not an eigenvalue and $\Phi$ is the unique $\ell^2$-normalized positive eigenfunction to $E_0$, otherwise.
Hence, we have
$$p_t (x,y) \longrightarrow \Phi (x) \Phi (y) = \frac{1}{m(X)}$$
and the corollary follows. The `in particular' follows, since we already know that $1\in D(Q^{(N)})$.
\end{proof}

\textbf{Remarks.}
(a) There are examples such that $m(X) <\infty$ and the form $Q^{(N)}$ is bounded, i.e., satisfies $Q^{(N)}(u,u) \leq \kappa  \|u\|^2$ for some $\kappa \geq 0$. More precisely, by \cite[Theorem~9.3]{HKLW}, this is exactly the case if there is $C\geq 0$ such that $\sum_{y}b(x,y)+c(x)\leq C m(x)$ for all $x\in X$. In this case   $Q^{(N)} = Q^{(D)}$. In particular, the previous corollaries then  apply to $Q^{(D)}$ as well.

(b) Even if $m(X) <\infty$ one can still have $p_t(x,y)\longrightarrow 0 $ as $t\longrightarrow \infty$ for $Q^{(D)}$.
Namely, there are examples such that  $m(X)< \infty $ and $1\not\in D(Q^{(D)})$ (see \cite[Section~4]{KL1} and \cite[Appendix~A]{HKLW}). Now, let $E_0$ denote the infimum of the spectrum of the associated operator $L^{(D)}$.  If $E_0>0$, then, by the already mentioned   analogue to a theorem of Chavel and Karp \cite[Theorem~8.1]{HKLW},  we have $p_t(x,y)\longrightarrow 0$ as $t\longrightarrow \infty$. Otherwise, if $E_0 =0$, then we have $p_t(x,y)\longrightarrow \Phi(x)\Phi(y)$. But since $0$ is not an eigenvalue (as $1\notin D(Q^{(D)})$) we conclude that $\Phi\equiv 0$ and once more $p_t (x,y) \longrightarrow 0$ follows.


\section{Probability:  Recurrence and stochastic completeness}\label{Section-Probability}

In this section, we discuss  a characterization for recurrence. We assume that $c\equiv0$ and let $(b,0)$ be a connected graph over $(X,m)$.

The heat semigroup $e^{-tL^{(D)}}$, $t\geq 0$, defined by the spectral theorem on $\ell^{2}(X,m)$ can be extended by monotone limits to a semigroup  on $\ell^{p}(X,m)$, $1\leq p\leq \infty$.

The form $Q^{(D)}$ is called \emph{recurrent} if
\begin{align*}
    \int_{0}^{\infty}e^{-tL^{(D)}}f(x)dt\in\{0,\infty\}
\end{align*}
for all $f\in \ell^{1}(X,m)$ with  $f\geq0$  and all $x\in X$.  The form $Q^{(D)}$ is called \emph{transient} if
\begin{align*}
    \int_{0}^{\infty}e^{-tL^{(D)}}f(x)dt<\infty,
\end{align*}
for all $f\in \ell^{1}(X,m)$ with  $f\geq0$  and all $x\in X$. By \cite[Lemma~1.6.4]{Fuk} a connected graph $b$ is either recurrent or transient.

 We call $Q^{(D)}$ \emph{stochastically complete} if
\begin{align*}
    e^{-t L^{(D)}}1 \equiv 1
\end{align*}
for some (all) $t > 0$, where $1$ is the constant function $1$ on $X$. Otherwise, we call $Q^{(D)}$ \emph{stochastically incomplete}. In \cite[Section~6]{HKLW} it was discussed that the definition of  stochastic completeness only makes sense for the Dirichlet Laplacian $L^{(D)}$.

In general, it is known that recurrence implies stochastic completeness which implies that $Q^{(D)} = Q^{(N)}$ (see \cite[Lemma 1.6.5]{Fuk} for the first implication and \cite[Corollary 5.3]{HKLW} for the second one).  We now discuss a result of \cite{Schm} giving   that, in the case of finite measure, these three properties are equivalent.

\begin{theorem}\label{t:probality}  (Theorem of M. Schmidt \cite{Schm}) If $(b,0)$ is connected and $m(X)<\infty$, then the following are equivalent:
\begin{itemize}
  \item [(i)] $Q^{(D)}$ is recurrent.
  \item [(ii)]$Q^{(D)}$ is stochastically complete.
  \item [(iii)] $Q^{(D)}=Q^{(N)}$.
\end{itemize}
\end{theorem}
\begin{proof}  The implications (i)$\Longrightarrow$(ii)$\Longrightarrow$(iii) follow from the references above and do not require finiteness of the measure. The implication (iii)$\Longrightarrow$(i) can be shown as follows: Due to the finiteness of the measure, the constant function $1$ belongs to $D (Q^{(N)})$. By (iii) it then belongs to $D (Q^{(D)})$ as well. As $D (Q^{(D)})$ is the closure of $C_c (X)$ with respect to $\aV{\cdot}_{Q^{(N)}}$  we obtain, in particular, that the constant function $1$  can be approximated  with respect to $\|\cdot\|_o$ by functions with compact support. Given this, (i) follows from the characterization of recurrence given in  Theorem~(3.63) of  \cite{Soa} (after observing that recurrence of a continuous time random process is equivalent to the recurrence of that process restricted to jumps only in discrete time). 
\end{proof}

\begin{corollary}\label{c:probability} If $(b,0)$ is {\cpt}, connected and $m(X)<\infty$, then $Q^{(D)}\neq Q^{(N)}$.  In particular, $Q^{(D)}$ is stochastically incomplete and transient.
\end{corollary}
\begin{proof} By Lemma~\ref{l:c_c} we conclude that $\oh u\vert_{\partial X}=0$ for all $u\in D(Q^{(D)})$. This implies that the constant functions are not in $D(Q^{(D)})$. However, they are in $D(Q^{(N)})$ by Theorem~\ref{t:finite_meassure}. Hence,  $Q^{(D)}\neq Q^{(N)}$. The second statement now follows by Theorem~\ref{t:probality}.
\end{proof}

\section{(Counter)examples, trees and outlook} \label{counter}
In this final section, we  have a closer look at trees and show that, in general, the implications $(A)\Longrightarrow (B)\Longrightarrow (C) \Longrightarrow (D)$ are not equivalences. We conclude with some open questions.

It turns out that in the case of trees, we have $\varrho^{2}=d$. This has strong consequences, in particular, $(A)\Longleftrightarrow (B)$ and $(C)\Longleftrightarrow(D)$. On the other hand, we show that $(C)$ does not imply $(B)$ even in this case.

Furthermore, we present (non-tree) examples showing that infinite  diameter with respect to $d$ does not imply infinite  diameter  with respect to $\varrho$ and, more importantly, that $(B)$ does not imply $(A)$ in general. We also show that the map $\iota:\overline{X}^{d}\longrightarrow\overline{X}^{\varrho}$ is not necessarily an embedding.

\subsection{Trees and $(A)\Longleftrightarrow(B)\not\Longleftarrow (C)\Longleftrightarrow(D)$}

In this subsection, we assume that $b$ induces a \emph{tree}, that is, $b$ is connected and there exists no  path $(x_{0},\ldots, x_{n})$ with $n\geq 2$ such that $b(x_{n},x_{0})>0$. Furthermore, we assume that $c\equiv 0$.

\begin{lemma}\label{l:trees} Let $(b,0)$ be a tree. Then, $\varrho^{2}=d$.
\end{lemma}
The proof of this lemma is rather immediate using the electrical network connection: In the context of  electrical networks, it is a well-known rule that,  when connecting networks in series, effective resistances add up. Thus, one figures that the  effective resistance $r$  equals $d$ on a finite tree. Now, using  Proposition~\ref{rho-free} and Theorem~\ref{t:rhofree} we obtain the statement of the lemma. However, as we omitted the proof of Proposition~\ref{rho-free}, we give a slightly different proof for the sake of being self-contained.

\begin{proof}[Proof of Lemma~\ref{l:trees}]
We know from Lemma~\ref{l:DLip} that $\varrho^2 \leq d$, so we now only  need to prove the converse direction.  Let $x,y\in X$ be arbitrary and let $(x_0,\dots,x_n)$ be a path with $x_0 = x$ and $x_n = y$. We construct a function $f$ as follows: we let $f(x_0)=0$, $f(x_j)- f(x_{j+1})= \frac{1}{b(x_j,x_{j+1})}$ if $x_j$ and $x_{j+1}$ are neighbors in the path, and $f(x)-f(y)=0$ if $x$ and $y$ are neighbors not on the path. Thus,
\begin{eqnarray*}
\widetilde{Q}(f) &=& \frac{1}{2} \sum_{x,y\in X} b(x,y)(f(x)-f(y))^2\\
&=& \sum_{j=1}^n b(x_j,x_{j-1})(f(x_j) - f(x_{j-1}))^2 \\
&=& \sum_{j=1}^n \frac{1}{b(x_j,x_{j-1})}.
\end{eqnarray*}
As the path above is the unique path connecting $x_0$ and $x_n$, the last quantity equals $d(x_0,x_n)$. As the function $f$ has finite $\widetilde{Q}$ energy, we have, by the definition of $\varrho$, that
\[\frac{|f(x_0)-f(x_n)|}{\widetilde{Q}(f)^{1/2} } \leq \varrho(x,y).\]
 Moreover, we obviously have
 $$|f(x_0)-f(x_n)|= f(x_n)-f(x_0) = \sum_{j=1}^n f(x_j)-f(x_{j-1})= d(x_0, x_n).$$
 Combining this with  the already shown $\ow Q (f) = d(x_0, x_n)$, we obtain from the previous inequality  $d^{1/2} (x_0, x_n) \leq \varrho (x_0, x_n)$
 and the proof is finished.
\end{proof}

Lemma~\ref{l:trees} immediately implies that the canonical map $\iota:\overline{X}^{d}\longrightarrow\overline{X}^{\varrho}$ is a homeomorphism and so we have the following corollary.

\begin{corollary}\label{c:trees1} Let $(b,0)$ be a tree. Then 
$$(A)\Longleftrightarrow(B).$$
\end{corollary}

A further consequence of Lemma~\ref{l:trees} is the next corollary.

\begin{corollary}\label{c:trees2} Let $(b,0)$ be a tree that is canonically compactifiable. Then, $(b,0)$ is locally finite.
\end{corollary}
\begin{proof} By Theorem~\ref{characterization-cpt}, for a graph being canonically compactifiable is equivalent to $\diam\overline{X}^{\varrho}<\infty$. By Lemma~\ref{l:trees} this is equivalent to $\diam\overline{X}^{d}<\infty$ in the case of trees. In particular, this implies that $b(x,y)\ge \eps$ for some $\eps>0$ (if $x_n,y_n\in X$ existed with $b(x_{n},y_{n})\longrightarrow 0$, then $\diam\overline{X}^{d}\ge d(x_{n},y_{n}) = 1/ b(x_{n},y_{n}) \longrightarrow\infty$ as ${n\to\infty}$). Hence, $b$ is locally finite by the assumption $\sum_{y\in X}b(x,y)<\infty$, $x\in X$.
\end{proof}

Note that the corollary above does not hold in the general case:
\begin{example}[$(D)\not\Longrightarrow$ locally finite] \label{example-one}
Let $(b,c)$ be a locally finite, {\cpt} graph over $X$. Pick a vertex $x\in X$ and an infinite set of vertices $y_{n}\in X$, $n\ge0$, such that $b(x,y_n)=0$. Define the graph $b'$ by setting $b'(x,y_{n})=b'(y_{n},x)>0$, $n\ge0$, such that $\sum_{n}b'(x,y_{n})<\infty$ and $b'(y,z)=b(y,z)$ for all other $y,z\in X$. The graph $(b',c)$ is  not locally finite as  $x$ has infinitely many neighbors. Let $\ow Q'$ be the form with respect to $(b',c)$ with generalized domain $\ow D'$. Since $b\le b'$ we have $\ow Q(f)\leq\ow Q'(f)$ for all $f\in \ow D'$ and, thus, $\ow D'\subseteq \ow D\subseteq \ell^{\infty}(X)$. In conclusion, $(b',c)$ is {\cpt} but not locally finite.
\end{example}

\begin{corollary}\label{c:trees3} Let $(b,0)$ be a tree. Then,
$$(C)\Longleftrightarrow(D).$$
\end{corollary}
\begin{proof} By Theorem~\ref{implication-D} it suffices to show $(D)\Longrightarrow(C)$. We assume $(D)$ which is equivalent to $\diam \overline{X}^{\varrho}<\infty$ by Theorem~\ref{characterization-cpt}. This, in turn, is  equivalent to
$$\diam \overline{X}^{d}<\infty$$
 in the case of trees by Lemma~\ref{l:trees}.  Moreover, by Corollary  \ref{c:trees2} we know local finiteness. Now,  $(C)$ follows immediately from Theorem~ \ref{diam-d-finite-implies-C}.
\end{proof}

Next, we give an example of a tree which is totally bounded for each intrinsic pseudo metric with respect to a finite measure but it is not totally bounded for $\varrho$.

\begin{example}[$(C)\not\Longrightarrow (B) $]\label{e:C_does_not_imply_B}
The example is an infinite comb with finite $d$ diameter, i.e., it can be thought of as a one sided infinite ray with weights $b$ increasing along the line (which gives it finite length) and from each vertex a copy of the same  one sided infinite ray with finite length emanates.

Let $X_n=\{x_{n,k}\mid k\ge0\}$, $n\ge0$, be one sided infinite rays and $X=\bigcup_{n\ge0}X_{n}$. We let the edge weight $b$ on $X_{n}$ be given by
\begin{align*}
    b(x_{n,k-1},x_{n,k})&=2^{k}, \quad \text{ for every }  k\ge1
\end{align*}
Moreover,
\begin{align*}
b(x_{n-1,0},x_{n,0})&=2^{n},\quad \text{ for every } n\ge 1.
\end{align*}
For all other pairs of vertices we let $b$ be zero.
Moreover, we let $m\equiv 1$ and $c\equiv 0$. We have
$\diam \overline{X_{n}}^{d}=1$ for $n\ge0$ and observe that $X$ cannot be covered by finitely many $d$ balls of radius $1$. Thus, $(A)$ is not satisfied. On the other hand, we have $\diam \overline{X}^{d}=3<\infty$ and by Corollary~\ref{sufficient-condition-cpt-d} we deduce $(D)$.
Hence, $(D)\not\Longrightarrow (A)$ and since the graph $(b,0)$ is a tree we have $(C)\not\Longrightarrow (B)$ by Corollaries~\ref{c:trees1} and~\ref{c:trees3}.
\end{example}

We end this subsection with a lemma which shows that a (strengthening of the) converse of Corollary~\ref{l:boundedness_sigma} is also true in the case of trees.

\begin{lemma}\label{l:boundedness_sigma_trees} Let $(b,0)$ be a tree over $(X,m)$. If
every intrinsic metric is bounded by $a^{-\frac{1}{2}}<\infty$ then
$$\inf_{x,y\in X, b(x,y)>0}\frac{b(x,y)}{m(x)\wedge m(y)}\ge a>0.$$
\end{lemma}
\begin{proof}
Let $x,y\in X$ with $b(x,y)>0$.
On a tree, a pseudometric $d_{\ell}$ given by a length function $\ell$ satisfies
\begin{align*}
    d_{\ell}(x,y)=\ell(x,y).
\end{align*}
 So, let $\ell$ be such that $\ell(x,y)=\ell(y,x)=(\frac{b(x,y)}{m(x)\wedge m(y)})^{-\frac{1}{2}}$ and $\ell\equiv0 $ on $X\times X\setminus\{(x,y),(y,x)\}$. It can be directly checked that $d_{\ell}$ is an intrinsic metric and, thus, bounded by  $ a^{-\frac{1}{2}}$ by assumption. Hence, assuming without loss of generality $m(x)\leq m(y)$ and using that $d_{\ell}$ is an intrinsic metric, we arrive at
 \begin{align*}
    \frac{b(x,y)}{m(x)\wedge m(y)}\geq a \frac{b(x,y)}{m(x)\wedge m(y)}d_{\ell}^{2}(x,y)= a.
 \end{align*}
\end{proof}


\subsection{$(B)\not\Longrightarrow (A)$}

In this subsection, we observe that, although $\varrho^2\leq d$ by Lemma~\ref{l:DLip}, there can be no upper bound of $d$ in terms of $\varrho$ when the graph is not a tree. In particular, we show that $(B)\not\Longrightarrow (A)$, and that the map $\iota:\overline{X}^{d}\longrightarrow\overline{X}^{\varrho} $ from Theorem~\ref{extension-d-to-varrho} is, in general, not injective. 




First, we present the example that shows $(B)\not\Longrightarrow (A)$. With the help of Theorem~\ref{characterization-cpt}, this example also serves as a counterexample to the converse in Corollary~\ref{sufficient-condition-cpt-d}, which states that  $\diam(\ov{X}^{d})<\infty$ implies that $X$ is {\cpt}.

\begin{example}[$(B)\not\Longrightarrow (A) $  and $\diam(\ov{X}^{\varrho})<\infty\not\Longrightarrow \diam(\ov{ X}^d) <\infty$]\label{e:diam_d_and_rho}
The example can be thought of as a one sided infinite ray, where we put over each edge an increasing number of additional paths in the form of triangles. It turns out that $d$ is not effected by these additional paths but $\varrho$ is.

Let $X=\{x_{n}\}_{n\ge1}\cup\bigcup_{n\ge1}\{x_{n}^{(1)},\ldots, x^{({n})}_{n}\}$ and let a symmetric $b $ be given such that
\begin{align*}
    2b(x_{n},x_{n+1})=
    b(x_{n},x_{n}^{(k)})=b(x_{n}^{(k)}, x_{n+1})=n,
\end{align*}
for $n\ge1$, $1\leq k\le {n}$,
and $b$ takes the value zero otherwise. Let further assume that $c\equiv0$ and $m\equiv1$.

We will show that $\overline {X}^{\varrho}\setminus X$ consists of one point and $X$ is totally bounded with respect to $\varrho$.

To show this, we can either calculate $\varrho(x_{1},x_{n})\leq 1/n^{2}$ directly using basic calculus and Lagrange multipliers or we use the electrical network connection and calculate  $r(x_{n},x_{n+1})$, $n\ge1$, first. Note that the free effective resistance $r(x_{n},x_{n+1})$ is smaller than or equal to the resistance between $x_{n}$ and $x_{n+1}$  on the subgraph induced by $\{x_{n},x_{n+1},x_{1}^{(1)},\ldots,x_{1}^{(n)}\}$. This, however, can be computed either using the rules for networks connected in series and in parallel \cite[Chapter~2.3]{LyonsBook} or by basic calculus to be $1/n^{2}$. Hence, $r(x_{n},x_{n+1})\leq 1/n^{2}$ and by Theorem~\ref{t:rhofree} and the triangle inequality for $r$
\begin{align*}
    \varrho(x_{1},x_{n+1})^\frac{1}{2}=r(x_{1},x_{n+1}) =\sum_{j=1}^{n}r(x_{j},x_{j+1})\leq
    \sum_{j=1}^{\infty}\frac{1}{j^{2}}.
\end{align*}
Moreover, $\varrho^{2}\leq d$ by Lemma~\ref{l:DLip}  and, therefore, $$\varrho(x_{n},x_{n}^{(k)})=\varrho(x_{n+1},x_{n}^{(k)}) \leq\frac{1}{\sqrt{2n}},$$ for all $n\ge1$ and $1\le k\le n$. Hence, $\diam(\overline {X}^{\varrho})<\infty$, the $\varrho$-boundary consists of one point and $X$ is totally bounded with respect to $\varrho$, that is, $(B)$ is fulfilled.

On the other hand, $$ d(x_{1},x_{n})=\sum_{j=1}^{n} \frac{1}{b(x_{j},x_{j+1})}=
2\sum_{j=1}^{n}\frac{1}{j}\to\infty,$$ as $n\to\infty$ and, therefore, $\diam(\overline {X}^{d})=\infty$.
As additionally $d(x_{1},x_{n}^{(k)})\ge d(x_{1},x_{n})$ for all $n\ge1$ and $1\leq k\leq n$, we  conclude  that $\overline {X}^{d}= X$. Hence, $(A)$ is not fulfilled, although $(B)$ is.
\end{example}

By similar means we can also show that $\iota:\overline{X}^{d}\longrightarrow\overline{X}^{\varrho} $ from Theorem~\ref{extension-d-to-varrho} is in general not an embedding (i.e., it is not injective).


\begin{example}[$\iota$ is not injective]\label{e:iota_no_embedding}
 We can think of the graph as two one sided infinite rays $(x^{(0)}_n)_n$ and $(x^{(1)}_n)_n$ where we attach to each pair $x^{(0)}_n, x^{(1)}_n$ of corresponding vertices additional paths via $n$ additional vertices.

Let $X=\{x_{n}^{(k)}\mid n\ge0, 0\le k\le n+1\}$ and let $b$ be symmetric and satisfy
\begin{align*}
    b(x^{(i)}_{n},x_{n+1}^{(i)})&=2^{n},\quad i=0,1,\\
    b(x^{(i)}_{n},x^{(j)}_{n})&=1,\quad i=0,1,j=0,\ldots,n+1, j\neq i.
\end{align*}
and let $b$ be zero otherwise as well as $c\equiv0$.
One sees immediately that the $d$ boundary of $X$ consists of two distinct points $\oh{x}^{(0)}$,  $\oh{x}^{(1)}$, which arise from the $d$ Cauchy sequences $(x_{n}^{(0)})_{n\ge0}$, $(x_{n}^{(1)})_{n\ge0}$ and  satisfy $d(\oh{x}^{(0)},\oh{x}^{(1)})=1$. Now, $\iota(\oh{x}^{(0)})$ and $\iota(\oh{x}^{(1)})$ belong to the $\varrho$ boundary of $X$. However, one can compute as in the example above  that $\varrho(x_{n}^{(0)},x_{n}^{(1)})\to0$, $n\to\infty$ and, therefore, $\iota(\oh{x}^{(0)})=\iota(\oh{x}^{(1)})$ with respect to $\varrho$. Thus, $\iota$ is not injective.
\end{example}

\subsection{Problems and outlook}
In this subsection, we discuss some open questions and problems raised by our considerations.

\bigskip

In the subsections above, we show that $(C)\not\Longrightarrow (B)\not\Longrightarrow (A)$. In Corollary~\ref{c:trees3} we have $(C)\Longleftrightarrow (D)$ for trees. This gives rise to the following question:

\begin{problem}Does $(D) $ imply $(C)$ in general?
\end{problem}

\medskip

On a different issue, it would be interesting to know whether the converse of Lemma~\ref{l:algebra} holds.

\begin{problem}
Does the fact that $\ow D$ is an algebra imply that $(b,c)$ is \cpt?
\end{problem}
To show this, one could try to present an unbounded $f\in \ow D$ such that $f^{2}\not\in \ow D$.

\medskip

Corollary~\ref{t:discrete_spectrum} tells us that being \cpt\ and having finite measure implies discrete spectrum, i.e., compact resolvent for the operator $L$ arising from any $Q^{(D)}\subseteq Q\subseteq Q^{(N)}$ . It would be interesting whether some kind of converse holds (see also the remark after Corollary~\ref{t:discrete_spectrum}).  
\begin{problem} Is $(b,c)$ {\cpt} if $L^{(N)}$ has compact resolvent for all $m$ with $m(X)<\infty$?
\end{problem}

We have shown that the Dirichlet problem can always be solved on the boundary of the Royden compactification $R$  under condition $(D)$ (Theorem~\ref{t:bvp}). It is natural to ask whether it can also be solved
 on the (larger) boundary of $\overline{X}^d$.
Here, we can point out the following: Example~\ref{e:iota_no_embedding} shows that $\iota : \overline{X}^d \longrightarrow K$ is, in general, not an embedding (even if $(A)$ holds).  This means that the functions of finite energy (which define the set $K$) will not separate points of  $\overline{X}^d \setminus X$. Thus, the  Dirichlet problem with the $d$-boundary of $X$ can, in general, not be solved with functions of finite energy (or uniform limits thereof). This does not preclude, however, the possibility that it can be solved within a larger class of functions (contained in the domain of the generalized Laplacian). This leads to the following problem.

\begin{problem} For which continuous functions on the boundary of $\ov{X}^d$ is it possible  to solve the Dirichlet problem if (D) holds?
\end{problem}

\medskip

In Subsection~\ref{s:Dirichlet_bc}, we showed that for \cpt\ graphs it becomes very clear why  $L^{(D)}$ is called  the Laplacian with Dirichlet boundary conditions. It seems to be essentially harder to even formulate a corresponding result for $L^{(N)}$, i.e., the Laplacian with Neumann boundary conditions. The reason is that one has to define the derivative on the boundary. Thus, we formulate the corresponding  question rather vaguely.

\begin{problem} What can be said about $L^{(N)}$ and $Q^{(N)}$ in view of Theorem~\ref{t:Dirichlet_bc}?
\end{problem}

\begin{appendix}

\section{Reducing summable $c\not \equiv 0$ to the case $c \equiv 0$}\label{Reducing}
In this section, we discuss how we can reduce the case of nonzero
summable  $c$ to the case $c\equiv0$. The basic idea is to add a virtual vertex to $X$ and extend functions from the domain of the form $\ow Q$ by zero at this vertex.

We assume here that $X$ is an infinitely countable set and $(b,c)$ is a weighted graph over $X$ which is connected.  Moreover, we assume  that $c$ is summable, i.e.,
\[\sum_{x\in X} c(x) < \infty.\]
Using this assumption, we add the element $\heartsuit$ to $X$, i.e., set $X_\heartsuit = X \cup \{\heartsuit\}$ and extend the weight $b$ to $b_\heartsuit$ on $X_\heartsuit \times X_\heartsuit$ by $b_\hrt(\hrt,\hrt)=0$ and
\[b_\heartsuit(x,\heartsuit)= c(x) = b_\heartsuit(\heartsuit, x).\]
By the assumption on $c$, this new graph satisfies the summability condition on the vertices. Consider the extended form
\[\widetilde{Q}_\heartsuit (f) = \frac{1}{2} \sum_{x,y\in X_\heartsuit} b_\heartsuit(x,y) |f(x) - f(y)|^2\]
defined for all functions $f$ on $X_\heartsuit$ (and, possibly, taking the value $\infty$). Define its  domain via
 $\widetilde{D}_\hrt = \{ f \in C(X_\hrt) : \ow{Q}_\hrt(f) < \infty \}$.  Let  $\LL_\hrt$ be the corresponding formal operator given by
\[ \LL_\hrt f(x) = \sum_{y \in X_\hrt} b_\hrt(x,y)(f(x) - f(y)) \]  which acts on $\ow{F}_\hrt = \{ f \in C(X_\hrt) : \sum_{y \in X_\hrt} b_\hrt(x,y) |f(y)| < \infty \mbox{ for all } x \in X_\hrt\}.$

 \smallskip

 \textbf{Remark.} The restriction of $\ow{F}_\hrt$ to $X$ just gives $\ow{F}$ and  for  $f\in \ow{F}_\hrt$ with $f (\hrt) =0$ we have
$\LL_\hrt f (x) = \LL f|_X (x)$ for all $x\in X$.

\smallskip

The next theorem shows that we can embed $\ow{D}$ into $\ow{D}_\heartsuit$. At the same time, it  demonstrates  the difference between these two spaces.

\begin{theorem}
The space $\widetilde{D}$ can be identified with a closed subspace of $\widetilde{D}_\heartsuit$, i.e., we have
\[\widetilde{D}=\{ f \in \widetilde{D}_\heartsuit : f(\heartsuit)=0\},\]
and
\[\widetilde{Q}_\heartsuit |_{\widetilde{D}} = \widetilde{Q}.\]
Furthermore, with respect to the inner product $\as{f,g}_\hrt = \ow{Q}_\hrt(f,g) + \ov{f(\hrt)}g(\hrt)$,  we have the orthogonal decomposition
\[\widetilde{D}_\heartsuit = \widetilde{D} \oplus H_\heartsuit\]
where $H_\heartsuit$ is a one dimensional space of functions $f$ which are harmonic on $X$, i.e.,  satisfy
\[\LL_\hrt f(x) =0\]
for $x\in X$.
\end{theorem}
\begin{proof}
The identification is clear. Note that, with $ \heartsuit$ instead of $o$  and $\ow{Q}_\hrt$ instead of $\ow Q$, we are exactly in the situation of Lemma \ref{l:pointevaluation}. Now, by this lemma,
convergence with respect to  $\aV{\cdot}_\heartsuit$ implies pointwise convergence and  the subspace  $\ow{D}$ is closed.  Thus, we have an orthogonal decomposition of $\widetilde{D}_\heartsuit$ into $\widetilde{D}$ and its orthogonal complement, which we denote by $H_\heartsuit$. We now show that
this space is one dimensional: Let $h_1,h_2 \in H_\heartsuit$ with $h_1(\heartsuit)=h_2(\heartsuit)=1$ be given.  Then $h_1-h_2 \in H_\heartsuit$ and $h_1-h_2 \in \widetilde{D}$. Thus $h_1-h_2 \equiv 0$ and we are done.

It remains to show that  functions from $H_\heartsuit$ are harmonic in $X$. This  follows from the definition of the orthogonal complement of $\widetilde{D}$ and an `integration by parts' formula as found in Subsection \ref{Quadratic} above, see \cite{HK, HKLW} for details as well. Specifically, this formula gives
$$\ow{Q}_\heartsuit (f,g) = \sum_{x\in X_\hrt}  \ov{\LL_\hrt f(x)} g(x)$$
for  all $f\in  \widetilde{D}_\heartsuit$ and  $g\in C_c (X_\heartsuit)$. Using this formula  for $f$ in the orthogonal complement of  $\ow{D}$  with respect to $\as{\cdot,\cdot}_\hrt$ and arbitrary $g\in C_c (X)$ (considered as elements in $\ow{D}\subset \ow{D}_\hrt$)  we then obtain
 \begin{eqnarray*} 0 &=& \as{f,g}_\hrt\\
  &=&  \ow{Q}_\hrt(f,g) + \ov{f(\hrt)} g(\hrt)\\
    &= & \sum_{x\in X_\hrt} \ov{\LL_\hrt f(x)} g(x)\\
     &= & \sum_{x\in X} \ov{\LL_\hrt f(x)} g(x).
     \end{eqnarray*}
Here, we used  $g(\hrt) =0$ in the last two steps.  As $g\in C_c (X)$ is arbitrary,
this easily gives the desired statement.
\end{proof}

\smallskip

\textbf{Remark.} The theorem asserts that any function in $H_\hrt$ is harmonic in $X$. If the graph  $(b_\hrt,0)$ on $X_\hrt$ is recurrent, we can also prove a converse: In that  case, $\ow{D}_\hrt$ is the closure of $C_c (X_\hrt)$ with respect to $\as{\cdot,\cdot}_\hrt$ by Theorem 3.63 of \cite{Soa}. It is not hard to see then that any $g\in \ow{D}_\hrt$  with $g(\hrt) =0$, i.e., any $g\in \ow{D}$  can be approximated by elements of $C_c (X_\hrt)$ with $g(\hrt) =0$ with  respect to $\as{\cdot,\cdot}_\hrt$. Let now $f\in \ow{D}_\hrt$ with $\LL_\hrt f (x) =0$ for all $x\in X$ be given. Reading the corresponding part of the previous proof backwards, we then find that
$0 = \as{f,g}_\hrt$ for all $g\in C_c (X)$. By the discussed denseness of these $g$, we obtain that $f$ is orthogonal to $\ow{D}$.

\bigskip

As an application, we have a look at the metric $\varrho$. We set
\[\varrho_\heartsuit(x,y)= \sup \{|f(x) - f(y)| : f\in \widetilde{D}_\heartsuit, \widetilde{Q}_\heartsuit(f) \leq 1 \}.\]
By $\varrho$ we denote the respective metric on the original graph $(b,c)$.
\begin{proposition}
For all $x,y \in X$ we have
\[\varrho(x,y) \leq \varrho_\heartsuit(x,y) \leq \varrho(x,y) + |f_H(x) - f_H(y)|\]
where $f_H$ is the unique (up to sign) function in $H_\heartsuit$ with $\widetilde{Q}_\heartsuit (f_H)=1$.
\end{proposition}
\begin{proof}
The first inequality is clear, as the set over which the supremum is taken is bigger. On the other hand, for $f\in \widetilde{D}_\heartsuit$ we use the decomposition $f=f_o + f_h$ from the previous theorem. As $f_o$ and $f_h$ are orthogonal with respect to $\as{\cdot, \cdot}_\heartsuit$, we obtain $\widetilde{Q}_\heartsuit(f_o,f_h)=0$. Thus, if $\widetilde{Q}_\heartsuit(f)\leq 1$, we have $\widetilde{Q}(f_o) \leq 1$ and $\widetilde{Q}_\heartsuit (f_h) \leq 1$. Taking suprema over the relevant sets we obtain the claim.
\end{proof}

As a second application, we consider length metrics.  We now consider the set of all paths over $X_\hrt$ and denote by
\[ d_\hrt(x,y) = \inf \{ \sum_{i=1}^n \frac{1}{b_\hrt(x_{i-1},x_i)} : (x_0, \ldots, x_n) \mbox{ is a path from } x \mbox{ to } y \} \]
where $d$ is the corresponding metric over paths in $X$ as before.
If $c>0$, then $d_c(x,y)$ is the infimum of lengths of paths connecting $x$ and $y$ in $X$, where the length of an edge is given by $\ell_c(x,y) = c(x)^{-1} + c(y)^{-1}$.
 As the set of paths joining $x$ and $y$ in the graph $X$ is smaller than the corresponding set in $X_\heartsuit$, we obtain
\[ d_\heartsuit (x,y) \leq d(x,y)\]
and
\[d_\heartsuit (x,y) \leq d_c(x,y).\]
Furthermore, we obtain the following improved version of Lemma \ref{l:DLip}.
\begin{lemma}
For all $f\in \widetilde{D}$ we have
\[ |f(x)-f(y)|^2 \leq \widetilde{Q}(f) d_\heartsuit(x,y).\]
In particular, we have
\[\varrho^2 \leq d_\heartsuit.\]
\end{lemma}
\begin{proof}
The proof of Lemma~\ref{l:DLip} shows that this inequality holds for all $f\in \widetilde{D}_\heartsuit$ with $\ow{Q}_\heartsuit(f)$ in place of $\ow{Q}(f)$. As $\widetilde{Q}_\heartsuit$ agrees with $\widetilde{Q}$ on $\widetilde{D}$ we obtain the inequality. The second claim follows from the proposition above.
\end{proof}

We finish this section with a short discussion concerning the case when $c$ is not summable. Let $\{X_i \mid i\in I\}$ be a partition of $X$, where $I$ is a countable set, such that, for all $i\in I$, we have $\sum_{x\in X_i} c(x) <\infty$. To each such $X_i$, we can adjoin an element $\heartsuit_i$. In particular, an analogous decomposition as above is valid. In the case of a finite $I$, one gets that the space of harmonic functions is $|I|$ dimensional. Denoting by $\varrho_I$ and $d_I$ the respective metrics on the supergraph, we obtain similar results as above for each partition. In particular, we can consider the infimum of these metrics over all partitions and we obtain also analogous versions of the previous results.

\section{Royden compactification for graphs with $c\not \equiv 0$}\label{A:Royden}
So far, the Royden compactification has only been considered for graphs with vanishing killing term $c$. Here, we show how to construct a compactification for general graphs $(b,c)$ over $X$. The construction is a rather direct modification of the construction of the Royden compactification for graphs of the form $(b,0)$ as presented, e.g., in \cite{Soa}.  For this reason, we continue to refer to the outcome as the
Royden compactification. We also shortly discuss how this (generalized) construction of the Royden compactification is stable under replacing $(b,0)$ by $(b,c)$ with summable $c$.

\medskip

Let  $(b,c)$ be a graph  over $X$ with associated form $\ow Q$ and consider the space
$$\mathcal{B}:={\ow D}\cap \ell^\infty (X)$$
 with the norm
$$\|u\|_{\ow Q, \infty}:= \ow Q (u) ^{1/2} + \|u\|_\infty. $$

The following lemma is well-known.

\begin{lemma} \label{estimate-product} For any $f,g\in \mathcal{B}$ the product $fg$ belongs to $\mathcal{B}$ and the  inequality
$$\ow Q (fg) \leq \left( \|f\|_\infty \ow Q (g)^{1/2} + \|g\|_\infty \ow Q (f)^{1/2}\right)^2$$
holds.
\end{lemma}
\begin{proof} It suffices to show the inequality (as the right hand side of it is obviously finite for $f,g\in \mathcal{B}$). A direct computation shows {\small
\begin{align*}
\lefteqn{ |  f (x) g (x) - f(y) g(y)|^2 = | f(x) (g(x) - g(y)) + g(y) (f(x) - f(y))|^2}\\
&\leq  |f(x)|^2 |g(x) - g (y)|^2 + 2 |f(x) g(y)| |f(x) - f(y)| |g(x) - g(y)| \\
 & \qquad + |g(y)|^2 |f(x) - f(y)|^2\\
&\leq   \|f\|_\infty^2 |g(x) - g (y)|^2 + 2 \|f\|_\infty \|g\|_\infty  |f(x) - f(y)| |g(x) - g(y)| \\
& \qquad + \|g\|_\infty^2 |f(x) - f(y)|^2
\end{align*}
}
Multiplication by $\frac{1}{2}b(x,y)$, summing over $x, y$, and addition of
$$\sum_{x\in X}  c(x) |f(x)|^2 |g(x)|^2  \leq \|f\|_\infty^2 \sum_x c(x) |g(x)|^2$$
 now easily  gives that $\ow Q (f g)$ is bounded above by
$$
  \|f\|_\infty^2 \ow Q (g) +  2 \|f\|_\infty \|g\|_\infty  \frac{1}{2} \sum_{x,y\in X} b(x,y) |f(x) - f(y)| |g(x) - g(y)| + \|g\|_\infty^2 \ow Q (f).$$
Now, an application of the Cauchy-Schwarz inequality gives the desired inequality.
\end{proof}

\begin{theorem} When equipped with $\aV{\cdot}_{\ow Q, \infty}$, the space  $\mathcal{B}$ becomes  a Banach algebra.
\end{theorem}

\textbf{Remark.} In the case $c\equiv0$, the algebra $\mathcal{B}$ is called the  \emph{Dirichlet algebra} \cite[Theorem~(6.2)]{Soa}.

\begin{proof} By the first part of the  previous lemma, the space $\mathcal{B}$ is an algebra. Moreover, the inequality in the  previous lemma easily gives that   $\aV{\cdot}_{\ow Q, \infty}$ is submultiplicative, i.e., satisfies
$$\|f g\|_{\ow Q,\infty} \leq \|f\|_{\ow Q, \infty} \|g\|_{\ow Q, \infty}.$$
It remains to show that $\mathcal{B}$ is complete with respect to $\aV{\cdot}_{\ow Q, \infty}$.
Let $(f_n)$ be a Cauchy sequence with respect to $\aV{\cdot}_{\ow Q, \infty}$. Then, it is a Cauchy sequence with respect to $\aV{\cdot}_\infty$.  Hence, it converges uniformly to a function $f$ on $X$. In particular, it converges pointwise to $f$. Now, an application of Fatou's lemma gives, for any natural number $n$, the estimate
$$\ow  Q (f - f_n) \leq \limsup_{m\longrightarrow \infty} \ow Q ( f_m - f_n) \leq \limsup_{m\longrightarrow \infty} \aV{ f_m - f_n}_{\ow Q, \infty}$$
and the desired statement follows.
\end{proof}

The previous theorem shows that we can associate to any graph $(b,c)$ over $X$ a commutative Banach algebra $(\mathcal{B}, \aV{\cdot}_{\ow Q, \infty})$.
We now apply the Gelfand theory of commutative Banach algebras. Recall that $C_0 (Y)$ denotes the space of continuous functions vanishing at infinity whenever $Y$ is a locally compact Hausdorff space.

\begin{theorem} \label{Characterization-Y}  Let $(b,c)$ be a graph over $X$ and $(\mathcal{B}, \aV{\cdot}_{\ow Q, \infty})$ the associated Banach algebra. Then, there exists a unique (up to homeomorphism) locally compact Hausdorff space $Y$ containing $X$  such that the following assertions  hold:

\begin{enumerate}

\item The set $X$ is  a dense  open subset of  $Y$.

\item Any function in $\mathcal{B}$ can be extended to an element of $C_0 (Y)$.
\item The algebra $\mathcal{B}$ separates points of $Y$.

\item The algebra $\mathcal{B}$ does not vanish on any point of $Y$.

\end{enumerate}

If the constant function $1$ belongs to $\mathcal{B}$, then  $Y$ is compact.

\end{theorem}
\begin{proof}
Note that, by the Stone-Weierstrass theorem, (2), (3) and (4) just say that the algebra $\mathcal{B}$ is dense in $C_0 (Y)$ with respect to $\aV{\cdot}_\infty$. This easily shows \textit{uniqueness}  of $Y$ up to homeomorphism (as $C_0 (Y)$  determines $Y$).

\smallskip

To show \textit{existence} we consider the set $Y_0$ of all multiplicative linear  functionals on $\mathcal{B}$.  This set is a compact space in the weak-$\ast$-topology and contains the zero functional $\gamma_\infty$. We will show that the locally  compact space  $Y:= Y_0 \setminus \{\gamma_\infty\}$ has the desired properties.

Obviously, $Y$  contains the set $X$ as any $x\in X$ provides a non-vanishing  multiplicative functional via point evaluation.
By the definition of the weak-$\ast$-topology,  the
 functions of $\mathcal{B}$ can be considered as   continuous functions on $Y_0$  which obviously vanish at $\gamma_\infty$. They therefore belong to $C_0 (Y)$.

 Moreover, as the characteristic function of any $x\in X$ is obviously an element of $\mathcal{B}$, it is continuous on $Y$ and hence any point $x\in X$ is an open set in $Y$.

By construction, the algebra $\mathcal{B}$ separates points of $Y$ and does not vanish on any point of $Y$. By the Stone-Weierstrass theorem, it is then dense in $C_0 (Y)$ and this implies that $X$ is dense in $Y$.

This shows that $Y$ indeed has all of the desired properties.

\smallskip

It remains to show the last statement. If the constant function $1$ belongs to $\mathcal{B}$, then the set $Y$ of all non-vanishing multiplicative functionals on $\mathcal{B}$ can easily be seen to be a  closed  subset of $Y_0$.  As a closed subset of a compact space, $Y$ is then compact as well.
\end{proof}

\textbf{Remarks.} (a) Note that the constant function $1$ belongs to $\mathcal{B}$ if and only if $c$ is summable.

(b) If the constant function $1$ belongs to $\mathcal{B}$, then property (4) of the above theorem is trivially satisfied and, in (3), the space $C_0 (Y)$ is just the space $C(Y)$.

(c) From (a) and (b) we see that  if $c\equiv0$, then the constant function 1 belongs to $\mathcal{B}$ and the compact space $Y$ is characterized by (1), (2) and (3). In this way, we recover Theorem~6.4 of \cite{Soa}.

(d) It is not hard to see that  $\mathcal{B}$ does not change if $(b,0)$ is replaced by $(b,c)$ with summable $c$. Moreover, the corresponding norms can easily be seen to be equivalent. Thus, the space $Y$ does not change when one replaces  $(b,0)$ by $(b,c)$ with summable $c$.

\medskip

\begin{definition} The \emph{Royden compactification} of the graph  $(b,c)$  over $X$ is defined to be the space $Y$ from the previous theorem if $c$ is summable and to be the one-point compactification of $Y$, otherwise.
\end{definition}

\end{appendix}


\begin{thebibliography}{99}


\bibitem{BJH} F. Bauer, B. Hua, J.  Jost,  The dual Cheeger constant and spectra of infinite graphs, preprint 2012. arXiv:1207.3410.

\bibitem{BHK} F. Bauer, B. Hua, M. Keller, On the $l^p$ spectrum of Laplacians on graphs, Advances in Mathematics {248} (2013), Issue  25,   717--735.

\bibitem{BKW} F. Bauer, M. Keller, R. Wojciechowski, Cheeger inequalities for unbounded graph Laplacians,  to appear in: Journal of the European Mathematical Society.


\bibitem{Car} R. Carlson, Boundary value problems for infinite metric graphs, in: Analysis on Graphs and Its Applications, Proc. Sympos. Pure Math., vol. 77, Amer. Math. Soc., Providence, RI, 2008, pp. 355--368.

\bibitem{CarAft} R. Carlson, After the explosion: Dirichlet Forms and
boundary problems for infinite graphs, preprint 2011.  arXiv:1109.3137.


\bibitem{CK} I. Chavel, L. Karp, Large time behavior of the heat kernel: the parabolic $\lambda$-potential alternative, Comment. Math. Helv. \textbf{66} (1991), no. 4,  541--556.

\bibitem{CdVTHT} Y. Colin de Verdi\`{e}re, N. Torki-Hamza,  F.  Truc, Essential self-adjointness for combinatorial Schr\"odinger operators II--metrically non complete graphs, Math. Phys. Anal. Geom. \textbf{14} (2011), no. 1,  21--38.


\bibitem{Davies} E. B. Davies, Heat kernels and spectral theory, Cambridge Tracts in Mathematics, vol. 92, Cambridge University Press, Cambridge, 1990.

\bibitem{Davies2} E. B. Davies, Analysis on graphs and noncommutative geometry, J. Funct. Anal.
\textbf{111} (1993), no. 2,  398--430.

\bibitem{Fol} M. Folz, Volume growth and stochastic completeness of graphs, to appear in: Transactions of the American Mathematical Society.

\bibitem{Fol2} M. Folz, Volume growth and spectrum for general graph Laplacians, to appear in: Mathematische Zeitschrift.

\bibitem{FLW} L. Frank, D. Lenz, D. Wingert,
{Intrinsic metrics for non-local symmetric Dirichlet forms
and applications to spectral theory}, to appear in: Journal of Functional Analysis.


\bibitem{Fuk} M. Fukushima, Y. {\=O}shima, M. Takeda, Dirichlet forms and symmetric Markov processes, de Gruyter Studies in Mathematics, vol. 19, Walter de Gruyter \& Co., Berlin, 1994.

\bibitem{Geo2} A. Georgakopoulos, Uniqueness of electical currents in a network of finite total resistance,
 J.\ London Math.\ Soc. {\bf 82} (2010), no. 1, 256--272.

\bibitem{Geo} A. Georgakopoulos, Graph topologies induced by edge lengths. In: Infinite Graphs: Introductions, Connections, Surveys. Special issue of Discrete Math., {\bf 311} (2011),  1523--1542.

\bibitem{GK} A. Georgakopoulos and K. Kolesko, Brownian motion on graph-like spaces. In preparation.

\bibitem{Gol} S. Golenia, Hardy inequality and eigenvalue asymptotics for discrete Laplacians, preprint 2012.  arXiv:1106.0658.





\bibitem{GHM}
A. Grigor'yan, X. Huang, J. Masamune, On stochastic completeness of jump processes, Math Z. \textbf{271} (2012), no. 3-4,  1211--1239.





\bibitem{HK} S. Haeseler, M. Keller, Generalized solutions and spectrum for Dirichlet forms on graphs, in  \cite{LSW}, 181--199.

\bibitem{HKLW} S. Haeseler, M. Keller, D. Lenz,  R.~K.~Wojciechowski, Laplacians on infinite graphs: Dirichlet and Neumann boundary conditions, J. Spectr. Theory \textbf{2} (2012), no. 4, 397--432.

\bibitem{HKW} S. Haeseler, M. Keller, R.~K.~Wojciechowski, Volume growth and bounds for the essential spectrum for Dirichlet forms,  to appear in: Journal of the London Mathematical Society.

\bibitem{HKT} M. Hinz, D. Kelleher, A. Teplyaev, Measures and Dirichlet forms under the Gelfand transform, preprint 2012. arXiv:1212.1099.

\bibitem{HKT2} M. Hinz, D. Kelleher, A. Teplyaev, Metrics and spectral triples for Dirichlet and resistance forms, preprint 2012.

\bibitem{HuK} B. Hua, M.  Keller, Harmonic functions of general graph Laplacians, to appear in Calc. Var.




\bibitem{Hua1} X. Huang, Stochastic incompleteness for graphs and weak Omori-Yau maximum principle, J. Math. Anal. Appl., \textbf{379} (2011), no. 2, 764--782.





\bibitem{HKMW} X. Huang, M.  Keller, J.  Masamune, R. Wojciechowski,  A note on self-adjoint extensions of the Laplacian on weighted graphs, J. Funct. Anal. \textbf{ 265} (2013), no. 8, 1556-1578.





\bibitem{Jor} P. E. T. Jorgensen, Essential selfadjointness of the graph-Laplacian, {J. Math. Phys.}, \textbf{ 49} (2008).


\bibitem{JP} P. E. T. Jorgensen, E. P. J. Pearse, Resistance boundaries of infinite networks, in: \cite{LSW}, 113--143.

\bibitem{JP2} P. E. T. Jorgensen, E. P. J. Pearse, Operator theory of electrical resistance networks,  to apper in: Springer Universitext.

\bibitem{JL} J. Jost, S. Liu, Ollivier's Ricci curvature, local clustering and curvature dimension inequalities on graphs, to appear in: Mathematical Research Letters.


\bibitem{KY} T. Kayano, M. Yamasaki, Some properties of Royden boundary of an infinite network, Mem. Fac. Sci. Shimane Univ. \textbf{22} (1988), 11--19.

\bibitem{Kel} M. Keller, The essential spectrum of the Laplacian on rapidly branching tessellations, Math. Ann. \textbf{346}  (2010), no. 1, 51--66.

\bibitem{KL1} M. Keller, D. Lenz, Dirichlet forms and stochastic completeness of graphs and subgraphs,  J. Reine Angew. Math. (Crelle's Journal) \textbf{666} (2012), 189--223.

\bibitem{KL2} M. Keller, D. Lenz, Unbounded Laplacians on graphs: basic spectral properties and the heat equation, Math. Model. Nat. Phenom. \textbf{5} (2010), no. 4,  198--224.

\bibitem{KL3} M. Keller, D. Lenz, Agmon type estimates and purely discrete spectrum for graphs, in preparation.

\bibitem{KLVW} M. Keller, D. Lenz, H. Vogt, R. Wojciechowski, Note on basic features of large time behaviour of heat kernels, to appear in: J. Reine Angew. Math. (Crelle's Journal).

\bibitem{KLW} M. Keller, D. Lenz, R. K. Wojciechowski, Volume growth, spectrum and stochastic completeness of infinite graphs, Math. Z. \textbf{274} (2013), no. 3-4, 905--932.

\bibitem{Kig01}J. Kigami, Analysis on fractals, Cambridge  Tracts in Mathematics, vol. 143, Cambridge University Press, Cambridge, 2001.

\bibitem{Kig} J. Kigami, Harmonic analysis for resistance forms, J. Funct. Anal. \textbf{204} (2003), no. 2,  399--444.

\bibitem{LSW} D. Lenz,  F. Sobieczky, W. Woess (eds.), Random Walks, Boundaries and Spectra, Progress in Probability, vol. 64, Birkh\"auser Verlag, Basel, 2011.

\bibitem{LStW} D. Lenz, P. Stollmann, D. Wingert, Compactness of Schr\"odinger semigroups,  Math. Nachr. \textbf{283} (2010), no. 1, 94--103.

\bibitem{Li} P. Li, Large time behavior of the heat equation on complete manifolds with nonnegative Ricci curvature, Ann. of Math. (2) \textbf{124} (1986), no. 1,  1--21.

\bibitem{LyonsBook} R.~Lyons with Y.~Peres, {\em Probability on Trees and Networks,} Cambridge University Press, to appear;
2013 draft at \verb+http://mypage.iu.edu/~rdlyons/prbtree/book.pdf+.

\bibitem{MU}
J. Masamune,  T. Uemura, Conservation property of symmetric jump processes.
Ann. Inst. Henri Poincar\'{e} Probab. Stat. {\bf 47} (2011), no. 3, 650--662.

\bibitem{MUW}
J. Masamune, T. Uemura,  J. Wang, On the conservativeness and the recurrence of symmetric jump-diffusions.
J. Funct. Anal. {\bf 263} (2012), no. 12,  3984--4008.


\bibitem{Roy} H. L. Royden, On the ideal boundary of a Riemann surface, Contributions to the theory of Riemann surfaces, Annals of Mathematics Studies, no. 30, Princeton University Press, Princeton, N. J., 1953, pp. 107--109.


\bibitem{Schm} M. Schmidt, Global properties of Dirichlet forms on discrete spaces, Diplomarbeit. arXiv:1201.3474v2.



\bibitem{Soa} P. M. Soardi, Potential theory on infinite networks, Lecture Notes in Mathematics, vol. 1590, Springer-Verlag, Berlin, 1994.

\bibitem{Stu} K.-T. Sturm, Analysis on local Dirichlet spaces. I. Recurrence, conservativeness and $L^p$ Liouville properties, J. Reine Angew. Math \textbf{456} (1994), 173--196.

\bibitem{TH} N. Torki-Hamza,  Laplaciens de graphes infinis I--Graphes m\'etriquement complets, Confluentes Math. \textbf{2} (2010), no. 3, 333--350.



\bibitem{Uem} T. Uemura, On symmetric stable-like processes: some path properties and generators. J. Theoret. Probab. \textbf{17} (2004), no. 3, 541--555.

\bibitem{Weidmann} J. Weidmann,  Linear operators in Hilbert spaces, Graduate Texts in Mathematics, vol. 68, Springer-Verlag, New York, 1980.

\bibitem{Woe}  W. Woess, Random walks on infinite graphs and groups, Cambridge Tracts in Mathematics, vol. 138, Cambridge University Press, Cambridge, 2000.

\bibitem{Woj1} R.~K.~Wojciechowski, Stochastic completeness of graphs, ProQuest LLC, Ann Arbor, MI, 2008.  Thesis (Ph.D.)--City University of New York.

\bibitem{Woj2} R.~K.~Wojciechowski, Heat kernel and essential spectrum of infinite graphs, Indiana Univ. Math. J. \textbf{58} (2009),  1419--1441.

\bibitem{Woj3} R.~K.~Wojciechowski, Stochastically incomplete manifolds and graphs,  in  \cite{LSW},  163--179.

\bibitem{Yam} M. Yamasaki, Discrete Dirichlet potentials on an infinite network, R.I.M.S. Kokyuroku \textbf{610} (1987), 51--66.

\bibitem{Yau} S. T. Yau, Some function-theoretic properties of complete Riemannian manifold and their application to geometry, Indiana Univ. Math. J. \textbf{25} (1976), no. 7, 659--670.

\end{thebibliography}
\end{document}